\documentclass[11pt,a4paper,reqno]{amsart}
\usepackage[english]{babel}
\usepackage[T1]{fontenc}
\usepackage{verbatim}
\usepackage{palatino}
\usepackage{amsmath}
\usepackage{mathabx}
\usepackage{amssymb}
\usepackage{amsthm}
\usepackage{amsfonts}
\usepackage{graphicx}
\usepackage{esint}
\usepackage{color}
\usepackage{mathtools}
\usepackage{overpic}

\usepackage[colorlinks = true, citecolor = black]{hyperref}
\author{Tuomas Orponen}
\title[Radial slices]{On the Hausdorff dimension of radial slices}
\address{Department of Mathematics and Statistics\\ University of Jyv\"askyl\"a,
P.O. Box 35 (MaD)\\
FI-40014 University of Jyv\"askyl\"a\\
Finland} 
\email{tuomas.t.orponen@jyu.fi}
\date{\today}
\thanks{T.O. is supported by the Research Council of Finland via the project \emph{Approximate incidence geometry}, grant no. 355453, and by the European Research Council (ERC) under the European Union’s Horizon Europe research and innovation programme (grant agreement No 101087499).}

\newcommand{\R}{\mathbb{R}}

\newcommand{\N}{\mathbb{N}}

\newcommand{\Z}{\mathbb{Z}}

\newcommand{\spt}{\operatorname{spt}}
\newcommand{\Hd}{\dim_{\mathrm{H}}}

\newcommand{\diam}{\operatorname{diam}}

\newcommand{\dist}{\operatorname{dist}}

\def\Barint_#1{\mathchoice
          {\mathop{\vrule width 6pt height 3 pt depth -2.5pt
                  \kern -8pt \intop}\nolimits_{#1}}%
          {\mathop{\vrule width 5pt height 3 pt depth -2.6pt
                  \kern -6pt \intop}\nolimits_{#1}}%
          {\mathop{\vrule width 5pt height 3 pt depth -2.6pt
                  \kern -6pt \intop}\nolimits_{#1}}%
          {\mathop{\vrule width 5pt height 3 pt depth -2.6pt
                  \kern -6pt \intop}\nolimits_{#1}}}

\numberwithin{equation}{section}

\theoremstyle{plain}
\newtheorem{thm}[equation]{Theorem}
\newtheorem*{"thm"}{"Theorem"}

\newtheorem{lemma}[equation]{Lemma}

\newtheorem{cor}[equation]{Corollary}
\newtheorem{proposition}[equation]{Proposition}

\theoremstyle{definition}

\newtheorem{definition}[equation]{Definition}

\theoremstyle{remark}
\newtheorem{remark}[equation]{Remark}

\newcommand{\nref}[1]{(\hyperref[#1]{#1})}

\DeclareMathSymbol{\intop}  {\mathop}{mathx}{"B3}

\begin{document}

\begin{abstract} Let $t \in (1,2)$, and let $B \subset \mathbb{R}^{2}$ be a Borel set with $\dim_{\mathrm{H}} B > t$. I show that
\begin{displaymath} \mathcal{H}^{1}(\{e \in S^{1} : \dim_{\mathrm{H}} (B \cap \ell_{x,e}) \geq t - 1\}) > 0 \end{displaymath}
for all $x \in \R^{2} \, \setminus \, E$, where $\dim_{\mathrm{H}} E \leq 2 - t$. This is the sharp bound for $\dim_{\mathrm{H}} E$.

The main technical tool is an incidence inequality of the form
\begin{displaymath} \mathcal{I}_{\delta}(\mu,\nu) \lesssim_{t} \delta \cdot \sqrt{I_{t}(\mu)I_{3 - t}(\nu)}, \qquad t \in (1,2), \end{displaymath}
where $\mu$ is a Borel measure on $\R^{2}$, and $\nu$ is a Borel measure on the set of lines in $\R^{2}$, and $\mathcal{I}_{\delta}(\mu,\nu)$ measures the $\delta$-incidences between $\mu$ and the lines parametrised by $\nu$.

This inequality can be viewed as a $\delta^{-\epsilon}$-free version of a recent incidence theorem due to Fu and Ren. The proof in this paper avoids the high-low method, and the induction-on-scales scheme responsible for the $\delta^{-\epsilon}$-factor in Fu and Ren's work. Instead, the inequality is deduced from the classical smoothing properties of the $X$-ray transform.
 \end{abstract}

\maketitle

\setcounter{tocdepth}{1}
\tableofcontents

\section{Introduction}

\subsection{Background on projections and slices} This paper studies a Marstrand-type slicing problem for radial projections. I start by recalling Marstrand's original work, and some subsequent developments. 

For $\theta \in [0,1]$, let $\pi_{\theta}(z) := z \cdot e_{\theta}$ denote the orthogonal projection of $z$ to the line spanned by $e_{\theta} := (\cos 2\pi \theta,\sin 2\pi \theta)$. Marstrand \cite{Mar} in 1954 proved the following theorem concerning the dimension of slices of fractal subsets of the plane. Let $B \subset \R^{2}$ be a Borel set with $\Hd B > t$ for some $t \in (1,2]$, or in fact already the weaker hypothesis $\mathcal{H}^{t}(B) > 0$ is sufficient. Here $\Hd$ denotes Hausdorff dimension, and $\mathcal{H}^{t}$ is $t$-dimensional Hausdorff measure. Then, for almost every $\theta \in [0,1]$,
\begin{equation}\label{form44} \mathcal{H}^{1}(\{r \in \R : \Hd (B \cap \pi_{\theta}^{-1}\{r\}) \geq t - 1\}) > 0. \end{equation}
In English, for almost every "angle" $\theta \in [0,1]$, positively many slices of $B$ by lines perpendicular to $e_{\theta}$ have Hausdorff dimension at least $t - 1$. In \cite{MR3145914}, I showed that the same conclusion actually holds for all $\theta \in [0,1] \, \setminus \, E$, where $\Hd E \leq 2 - t$. Results of this kind are known as \emph{exceptional set estimates}.

The threshold "$2 - t$" is sharp, because it is already sharp for a weaker theorem concerning projections. Namely, consider any $\theta \in [0,1]$ such that \eqref{form44} holds. Then clearly 
\begin{equation}\label{form45} \mathcal{H}^{1}(\pi_{\theta}(B)) > 0. \end{equation}
Indeed, \eqref{form45} says that positively many slices $B \cap \pi_{\theta}^{-1}\{r\}$ are non-empty, whereas \eqref{form44} adds information about their dimension. The best that one can say about the easier problem \eqref{form45} was already established by Falconer \cite{MR0673510} in 1982: provided that $\Hd B \geq t$ and $B$ is Borel, \eqref{form45} holds for all $\theta \in [0,1] \, \setminus \, E$ where $\Hd E \leq 2 - t$. This is sharp: the original construction is due to A. Peltom\"aki \cite{Peltomaki} from 1988, but for those readers whose Finnish is rusty, another source is \cite[Example 5.13]{MR3617376}.

While orthogonal projections are perhaps the most iconic projections in $\R^{2}$, another nice family consists of the \emph{radial projections} $\pi_{x} \colon \R^{2} \, \setminus \, \{x\} \to S^{1}$, defined by
\begin{displaymath} \pi_{x}(y) = \frac{x - y}{|x - y|}, \qquad y \in \R^{2} \, \setminus \, \{x\}. \end{displaymath} 
There is plenty of recent research \cite{2022arXiv220803597B,2023arXiv231105127B,MR4269398,2020arXiv200102551L,MR3503710,MR3778538,MR3892404,OSW22,2023arXiv230904097R,MR4634691} related to the following agenda: take a theorem which is known for orthogonal projections, involving Hausdorff dimension or measures, and prove its sharp counterpart for radial projections. Some of the listed papers also deal with this problem in finite fields and higher dimensional Euclidean spaces.

Projection and slicing problems for radial projections are formally harder than their counterparts for orthogonal projections: the family of orthogonal projections can be transformed, in an incidence preserving way, to the family of radial projections to any fixed line $\ell \subset \R^{2}$. For a more careful explanation, see \cite[Section 1.2]{OSW22}. One might therefore expect that if one replaces the orthogonal projections $\pi_{\theta}$ by radial projections $\pi_{x}$ in either \eqref{form44} or \eqref{form45} (or other problems of the same nature), there will be significantly more "exceptional" parameters $x \in \R^{2}$ than "exceptional" angles $\theta \in [0,1]$.

Surprisingly, the work in the papers listed above suggests the opposite: in some cases it has been shown that the sharp bound for exceptional radial projections is no larger than the sharp bound for exceptional orthogonal projections.

For example, as already mentioned, Falconer \cite{MR0673510} in 1982 showed that \eqref{form45} can fail for an at most $(2 - t)$-dimensional set of angles $\theta \in [0,1]$. In \cite{MR3892404}, I showed that the same remains true for radial projections: if $B \subset \R^{2}$ is a Borel set with $\Hd B > t \in (1,2]$, then 
\begin{equation}\label{form46} \mathcal{H}^{1}(\pi_{x}(B)) > 0 \end{equation}
for all $x \in \R^{2} \, \setminus \, E$, where $\Hd E \leq 2 - t$. The sharpness of this estimate follows from the sharpness of Falconer's estimate, and the formal connection between orthogonal and radial projections, see \cite[Section 3.2]{MR3892404} for the details.

\subsection{The radial slicing problem} Recall that \eqref{form44} is stronger than \eqref{form45}. Regardless, as shown in \cite{MR3145914}, the sharp exceptional set estimate for the problem \eqref{form44} is the same as the sharp exceptional set estimate for \eqref{form45} -- namely $\Hd E \leq 2 - t$ in both cases. 

How about the slicing version of the radial projection theorem \eqref{form46}? Given a Borel set $B \subset \R^{2}$ with $\Hd B > t \in (1,2]$, for how many parameters $x \in \R^{2}$ can the following fail:
\begin{equation}\label{form47} \mathcal{H}^{1}(\{e \in S^{1} : \Hd (B \cap \pi_{x}^{-1}\{e\}) \geq t - 1\}) > 0? \end{equation}
Note that \eqref{form46} is the radial projection counterpart of \eqref{form44}, and is stronger than \eqref{form46}. Marstrand \cite{Mar} originally proved that \eqref{form46} holds for $\mathcal{H}^{t}$ almost all $x \in B$. 

In joint work \cite{MR3503710} with Mattila, we showed that \eqref{form47} can fail for an at most $1$-dimensional set of parameters $x \in \R^{2}$. We also proposed in \cite[Question 1.4]{MR3503710} that the sharp number should be $2 - t$. The main result of the present paper verifies this hypothesis:

\begin{thm}\label{thm2} Let $t \in (1,2)$, and let $B \subset \R^{2}$ be a Borel set with $\Hd B > t$. Then \eqref{form47} holds for all $x \in \R^{2} \, \setminus \, E$, where $\Hd E \leq 2 - t$.  \end{thm}

\begin{remark}\label{rem2} Theorem \ref{thm2} (re-)covers the sharp exceptional set estimates for all the problems \eqref{form44}, \eqref{form45}, \eqref{form46}, and \eqref{form46} at the same time, but still remains slightly unsatisfactory. Namely, as in Marstrand's original slicing theorem \cite{Mar}, one might hope to prove Theorem \ref{thm2} under the weaker hypothesis that $\mathcal{H}^{t}(B) > 0$. This remains open. \end{remark}

\subsection{An incidence estimate}  The main new tool for proving Theorem \ref{thm2} is the following weighted incidence inequality:

\begin{thm}\label{main2} Let $\mu$ be a finite Borel measure on $B(1)$, and let $\nu$ be a finite Borel measure on $[\tfrac{1}{4},\tfrac{3}{4}] \times [-1,1]$. Let $\delta \in (0,1/50]$. Then,
\begin{equation}\label{form48} \mathcal{I}_{\delta}(\mu,\nu) \lesssim_{t} \delta \cdot \sqrt{I_{3 - t}(\mu)I_{t}(\nu)}, \qquad t \in (1,2), \end{equation}
where $I_{s}(\mu) = \iint |x - y|^{-s} \, d\mu x \, d\mu y$ is the $s$-dimensional Riesz energy for $s \in (0,2)$. \end{thm}
The notation will be properly defined in Section \ref{s:incidences}. Here I just mention that the measure $\nu$ should be interpreted as being supported on the set of lines in $\R^{2}$, where the lines are parametrised by $[0,1] \times \R$. In fact, the parametrisation is concretely $(\theta,r) \mapsto \{z \in \R^{2} : z \cdot (\cos 2\pi \theta,\sin 2\pi \theta) = r\}$. The quantity $\mathcal{I}_{\delta}(\mu,\nu)$ is defined as
\begin{equation}\label{form51} \mathcal{I}_{\delta}(\mu,\nu) := (\mu \times \nu)(\{(p,q) : p \in T_{q}(\delta)\}), \end{equation}
and $T_{q}(\delta)$ is the $\delta$-tube around the line with parameter $q$.

Theorem \ref{main2} is closely related to the incidence estimates obtained recently by Fu and Ren \cite{fu2022incidence}. In fact, a version of Theorem \ref{main2} with $\delta^{-\epsilon}$-losses could be deduced from \cite[Theorem 1.5]{fu2022incidence}. However, the proof of Theorem \ref{main2} is completely different from the proof in \cite{fu2022incidence}. The $\delta^{-\epsilon}$-factor in \cite{fu2022incidence} arises from an induction-on-scales argument, whereas \eqref{form48} is deduced from the smoothing properties of the planar $X$-ray transform, see Section \ref{s:outline} for a quick explanation. In the converse direction, Theorem \ref{main2} does not seem to imply \cite[Theorem 1.5]{fu2022incidence} in full generality (even if $\delta^{-\epsilon}$-losses are allowed).

\begin{remark} The $\delta^{-\epsilon}$-freeness of Theorem \ref{main2} is undeniably useful for the argument presented below for Theorem \ref{thm2}. Even a factor of order $\log(1/\delta)$ would be unacceptable in the final estimates below \eqref{form39}. This was indeed the main reason to seek a version of Fu and Ren's estimates without the $\delta^{-\epsilon}$-factor. Nonetheless, I do not have a convincing philosophical reason why the $\delta^{-\epsilon}$-freeness would be strictly necessary for Theorem \ref{thm2}. In a related context, Harris \cite[Theorem 4.2]{2023arXiv230104645H} proves a Marstrand-type slicing theorem for vertical projections in the first Heisenberg group, and the proof in that paper is crucially based on the endpoint ($\delta^{-\epsilon}$-free) trilinear Kakeya inequality. 
 \end{remark}

\subsection{Proof ideas and challenges}\label{s:outline} The full details of the incidence estimate \eqref{form48} are given in Section \ref{s:XRay}. Here we ignore the technicalities and just present the simple idea. For $f \in C_{c}([0,1] \times \R)$, consider the adjoint $X$-ray transform
\begin{displaymath} R^{\ast}f(z) := \int_{0}^{1} f(\theta,\pi_{\theta}(z)) \, d\theta. \end{displaymath}
If $f$ is interpreted as a density on the set of lines, and $g \in C_{c}(\R^{2})$ is a density on the set of points, then the quantity 
\begin{equation}\label{form50} "\mathcal{I}(f,g)" := \iint g(z)R^{\ast}f(z) \, dz = \iint Rg(\theta,r)f(\theta,r) \, d\theta \, dr \end{equation}
arguably measures the incidences between the lines parametrised by $f$, and $g$. 

Moving towards \eqref{form48}, assume that $f$ is a fractal measure on the space of lines, say $I_{t}(f) \leq 1$ for $t \in (0,2)$ (here $I_{t}(f)$ is the $t$-dimensional Riesz energy). Then $f \in H^{-(2 - t)/2}$. By the duality of $H^{s}$ and $H^{-s}$, the right hand side of \eqref{form50} is bounded whenever $Rg \in H^{(2 - t)/2}$. By the well-known $\tfrac{1}{2}$-order smoothing behaviour of the operator $R$, this is true if 
\begin{displaymath} g \in H^{[(2 - t)/2] - 1/2} = H^{-[2 - (3 - t)]/2} \quad \Longleftrightarrow \quad I_{3 - t}(g) < \infty. \end{displaymath}
The hypothesis $t > 1$ was needed in the final equivalence. To summarise, the quantity "$\mathcal{I}(f,g)$" will be bounded if $I_{t}(f) < \infty$ and $I_{3 - t}(g) < \infty$. This is \eqref{form48}. The additional factor "$\delta$" appears when one honestly expresses the quantity $\mathcal{I}_{\delta}(\mu,\nu)$ from \eqref{form51} in the form \eqref{form50}, see the argument leading to \eqref{form52}.

Unfortunately, the proof of Theorem \ref{thm2} is not as clean-cut. The first step is to find a suitable $\delta$-discretised version, and reduce matters to that version. This is accomplished in Section \ref{s:discretisation}. Roughly speaking, Theorem \ref{thm3} says the following.

Assume that $E,F \subset B(1)$ are a pair of $\delta$-discretised sets with $\dist(E,F) \geq 1$, where $E$ is $(2 - t)$-dimensional and $F$ is $t$-dimensional. Assume that to each $p \in E$ there corresponds a family $\mathcal{T}_{p}$ of $\delta$-tubes covering $99\%$ of $F$. Then, for every $\tau < t$, there exists a tube $T \in \mathcal{T}_{p}$ (for some $p \in E$) such that $\mathcal{H}^{\tau - 1}_{\delta,\infty}(F \cap T) \gtrsim_{\tau} 1$. Proving this statement with $\tau = t$ would likely solve the problem mentioned in Remark \ref{rem2}.

The proof of Theorem \ref{thm3} proceeds by making a counter assumption: $\mathcal{H}^{\tau - 1}_{\delta,\infty}(F \cap T) \leq \epsilon$ for all $T \in \mathcal{T}_{p}$ and $p \in E$. This form of the counter assumption forebodes a difficulty: life would be easier if one could assume $\mathcal{H}^{\tau - 1}_{\delta,\infty}(F \cap T) \leq \delta^{\epsilon}$. I do not know how to reduce matters to this easier problem. This issue led to the introduction of Theorem \ref{main2}. With only the weaker form of the counter assumption available, the proof of Theorem \ref{thm3} may not introduce any factors which grow as functions of $\delta^{-1}$, such as $\log(1/\delta)$. Recalling that Theorem \ref{thm3} is a statement about $\delta$-tubes, this is a non-trivial problem.

With some initial pigeonholing, the counter assumption allows one to find an intermediate scale $\Delta \in (\delta,\tfrac{1}{2}]$ with $\Delta = o_{\epsilon}(1)$, where the intersections $F \cap T$ look "over-crowded" for many tubes $T \in \mathcal{T} := \bigcup_{p \in E} \mathcal{T}_{p}$. This pigeonholing only produces factors of the order $\log(1/\Delta)$, which are manageable. It may be helpful to think that $\Delta \sim \epsilon$ from now on.

Above the words "over-crowded" mean that tube-segments of the form $T \cap B(x_{0},\Delta)$ are suspiciously rich with points from $F \cap B(x_{0},\Delta)$. Here $B(x_{0},\Delta)$ is a judiciously chosen but fixed $\Delta$-disc. With such information in hand, one is tempted to apply Theorem \ref{main2} to (suitable measures associated to)
\begin{displaymath} F \cap B(x_{0},\Delta) \quad \text{and} \quad \mathcal{T} \cap F \cap B(x_{0},\Delta), \end{displaymath}
where $\mathcal{T} \cap F \cap B(x_{0},\Delta) = \{T \in \mathcal{T} : T \cap F \cap B(x_{0},\Delta) \neq \emptyset\}$. The good news is that $F \cap B(x,\Delta)$ is still a $t$-dimensional set. The bad news is that there is no \emph{a priori} information about the family $\mathcal{T} \cap F \cap B(x_{0},\Delta)$. Fortunately, the statement of Theorem \ref{main2} guides us in the right direction: we just need to show that $\mathcal{T} \cap B(x_{0},\Delta)$ is a $(3 - t)$-dimensional set, or at least contains a big $(3 - t)$-dimensional subset.

The main observation is that the tube sets $\mathcal{T}^{q} = \{T \in \mathcal{T} : q \in T\}$ are roughly $(2 - t)$-dimensional, because $\mathcal{T}^{q} \cong \pi_{q}(E)$, and $E$ was assumed to be $(2 - t)$-dimensional. Consequently the family 
\begin{displaymath} \mathcal{T}' := \bigcup_{q \in F \cap B(x_{0},\Delta)} \mathcal{T}^{q} \subset \mathcal{T} \cap F \cap B(x_{0},\Delta) \end{displaymath}
is a (dual) $(2 -t,t)$-Furstenberg set, and has dimension $\geq (3 - t)$ by \cite[Theorem 1.6]{fu2022incidence}. This is what was needed. Once it has been established that $F \cap B(x_{0},\Delta)$ is $t$-dimensional and $\mathcal{T}'$ is $(3 - t)$-dimensional, Theorem \ref{main2} produces an upper bound for incidences which contradicts the "over-crowding" phenomenon for the tube segments $T \cap B(x_{0},\Delta)$.

The informal arguments with "dimension" in the previous paragraph need to be quantified carefully. This work may have some independent interest. For example, Theorem \ref{thm3a} is a quantitative version of \cite[Theorem 1.6]{fu2022incidence}, and the proof is rather different from Fu and Ren's original argument. Concretely, in the proof Theorem \ref{thm3}, one has access to information of the form "$\mathcal{T}^{q}$ is a $(\delta,2 - t,\log(1/\Delta))$-set". Then, to proceed, one infers from Theorem \ref{thm3a} that $\mathcal{T}'$ contains a $(\delta,3 - t,\log^{10}(1/\Delta))$-set.

\subsection*{Notation} I write $A \lesssim B$ if there exists an absolute constant $C > 0$ such that $A \leq CB$. If $C$ is allowed to depend on a parameter "$p$", this is signified by writing $A \lesssim_{p} B$. Later, it will be agreed that the notation "$\lesssim_{p}$" is abbreviated to "$\lesssim$" for certain key parameters "$p$".

\subsection*{Acknowledgements} I am grateful to Pertti Mattila for introducing me to the problem studied in Theorem \ref{thm2}, and for numerous fruitful discussions on the topic.

\section{Incidence estimates via the $X$-ray transform}\label{s:XRay}

In this section Theorem \ref{main2} is proved. The argument is mainly based on $\tfrac{1}{2}$-order Sobolev smoothing property of the $X$-ray transform in the plane, which is classical and very well documented, see for example \cite[Theorem 5.3]{MR0856916}. I decided to present the necessary theory without any prerequisites assumed on the $X$-ray transform, but experts on this topic are advised to skip ahead to Section \ref{s:incidences}. 

\subsection{The $X$-ray transform} We start with the definitions.

\begin{definition}[$X$-ray transform] For $\theta \in [0,1]$, let $e_{\theta} := (\cos 2\pi \theta,\sin 2\pi \theta)$, and let $\pi_{\theta}(z) := z \cdot e_{\theta}$ be the orthogonal projection to the line spanned by $\theta$. For $g \in C_{c}(\R^{2})$ or $g \in \mathcal{S}(\R^{2})$, the \emph{$X$-ray transform} of $g$ is defined by
\begin{displaymath} Rg(\theta,r) := \int_{\pi_{\theta}^{-1}\{r\}} g(z) \, d\mathcal{H}^{1}(z), \qquad (\theta,r) \in [0,1] \times \R. \end{displaymath}
Thus, $R$ maps $g$ to a continuous function defined on the "space of lines" parametrised by the pair $(\theta,r)$. For $f \in C_{c}([0,1] \times \R)$, consider also the \emph{adjoint $X$-ray transform}
\begin{displaymath} R^{\ast}f(z) := \int_{0}^{1} f(\theta,\pi_{\theta}(z)) \, d\theta, \qquad z \in \R^{2}. \end{displaymath} \end{definition}

\begin{remark} As the naming suggests, the operators $R,R^{\ast}$ indeed satisfy the following formal adjointness property: given $f \in C_{c}([0,1] \times \R)$ and $g \in C_{c}(\R^{2})$,
\begin{align*} \int_{\R^{2}} R^{\ast}f(z)g(z) \, dz & = \int_{\R^{2}} \left(\int_{0}^{1} f(\theta,\pi_{\theta}(z)) \, d\theta \right) g(z) \, dz \\
& = \int_{0}^{1} \int_{\R} f(\theta,r) \int_{\pi_{\theta}^{-1}\{r\}} g(z) \, d\mathcal{H}^{1}(z) \, dr \, d\theta \\
& = \int_{0}^{1} \int_{\R} f(\theta,r)(Rg)(\theta,r) \, dr \, d\theta. \end{align*} \end{remark}

\begin{remark} Note that $Rg(0,\cdot) = Rg(1,\cdot)$, so $Rg$ can also be viewed as a $1$-periodic function in the $1^{st}$ variable. This will be useful to keep in mind when we are later computing the Fourier coefficients of $Rg$ in the $1^{st}$ variable.

\end{remark}

The famous "projection slice-theorem" says that the $X$-ray transform has a neat relation to the Fourier transform. We record this below. For $h \in C_{c}([0,1] \times \R)$, let
\begin{equation}\label{form53} \tilde{h}(\theta,\rho) := \int_{\R} e^{-2\pi i \rho r} h(\theta,r) \, dr, \qquad (\theta,\rho) \in [0,1] \times \R, \end{equation} 
be the Fourier transform of $h$ in (only) the $2^{nd}$ variable. Now, if $g \in C_{c}(\R^{2})$, we have $Rg \in C_{c}([0,1] \times \R)$, and
\begin{align} \widetilde{(Rg)}(\theta,\rho) & = \int_{\R} e^{-2\pi i \rho r} \int_{\pi_{\theta}^{-1}\{r\}} g(z) \, d\mathcal{H}^{1}(z) \, dr \notag\\
&\label{form3} = \int_{\R^{2}} e^{-2\pi i z \cdot \rho e_{\theta}} g(z) \, dz = \hat{g}(\rho e_{\theta}), \qquad (\theta,\rho) \in [0,1] \times \R, \end{align} 
where $\hat{g}$ is the full Fourier transform of $g$ in $\R^{2}$.

\subsection{Angular derivative of the $X$-ray transform} The "angular" $\partial_{\theta}$-derivative of $Rg$ coincides with the "radial" $\partial_{r}$-derivative of $R(Pg)$, where $P$ is a first degree polynomial, \cite[p. 12]{MR0856916}. I repeat the calculation from the reference, using the coordinates above for the $X$-ray transform. Let $g \in \mathcal{S}(\R^{2})$, and following \cite[p. 12]{MR0856916}, write $Rg$ in the following slightly imprecise way:
\begin{displaymath} Rg(\theta,r) = \int g(z)\delta(r - z \cdot e_{\theta}) \, dz, \end{displaymath}
where $\delta$ stands for the Dirac $\delta$ on $\R$. This could be made rigorous by replacing $\delta$ by an approximate identity, see \cite[p. 12]{MR0856916}. Pretending that one may compute the $\partial_{\theta},\partial_{r}$ derivatives of $(\theta,r) \mapsto \delta(r - z \cdot e_{\theta})$, one arrives a little formally at the following useful relation, where $\bar{e}_{\theta} := (\sin 2\pi \theta,-\cos 2\pi \theta)$:
\begin{align} \partial_{\theta}(Rg)(\theta,r) & = \int g(z)[\partial_{\theta}\delta(r - z \cdot e_{\theta})] \, dz \notag\\
& = \int g(z)(z \cdot \bar{e}_{\theta}) \delta'(r - z \cdot e_{\theta}) \, dz \notag\\
& = \bar{e}_{\theta} \cdot \partial_{r}\int g(z)z\delta(r - z \cdot e_{\theta}) \, dz \notag\\
&\label{form2a} = \bar{e}_{\theta} \cdot (\partial_{r}R(zg))(\theta,r), \qquad (\theta,r) \in (0,1) \times \R. \end{align}
Note that $zg \in \mathcal{S}(\R^{2},\R^{2})$ is a vector-valued Schwartz function with components $z_{1}g$ and $z_{2}g$. The notation $R(zg)$ refers to the $\R^{2}$-valued function with components $R(z_{1}g)$ and $R(z_{2}g)$, and the $\partial_{r}$-derivative is calculated component-wise. We leave it to the reader (of \cite[p. 12]{MR0856916}) to justify \eqref{form2a} rigorously for $g \in \mathcal{S}(\R^{2})$.

\subsection{Homogeneous Sobolev norms} I next survey the $\tfrac{1}{2}$-order smoothing behaviour of the $X$-ray transform. This section can be viewed as an exposition of \cite[Theorem 5.3]{MR0856916} albeit with slightly different notational conventions. 

\begin{definition}[Fourier transforms and Sobolev norms] For $f \in L^{1}([0,1] \times \R)$, define the (full) Fourier transform
\begin{displaymath} (\mathcal{F}f)(n,\rho) := \int_{0}^{1}\int_{\R} e^{-2\pi i(n \theta + \rho r)} f(\theta,r) \, d\rho \, d\theta, \qquad (n,\theta) \in \Z \times \R. \end{displaymath} 
For $s \geq 0$, define the (homogeneous) Sobolev norm
\begin{displaymath} \|f\|^{2}_{H_{0}^{s}} := \sum_{n \in \Z} \int_{\R} |(\mathcal{F}f)(n,\rho)|^{2}|(n,\rho)|^{2s} \, d\rho, \qquad f \in L^{1}([0,1] \times \R). \end{displaymath}
(The possibility $\|f\|_{H_{0}^{s}}^{2} = \infty$ is allowed.) For $g \in L^{1}(\R^{2})$ and $s \in \R$, we also denote
\begin{displaymath} \|g\|_{H^{s}_{0}}^{2} := \int_{\R^{2}} |\hat{g}(\xi)|^{2}|\xi|^{2s} \, d\xi. \end{displaymath}
The correct interpretation of $\|h\|_{H^{s}_{0}}$ is always clear from the domain of $h$.  \end{definition}

\begin{remark} A common trick below will be to use Plancherel in (only) the $1^{st}$ variable. If $g \in C_{c}([0,1] \times \R)$ is a function with $g(0,\cdot) = g(1,\cdot)$, then
\begin{equation}\label{form54} \sum_{n \in \Z} \int_{\R} |(\mathcal{F}g)(n,\rho)|^{2} \, d\rho = \int_{0}^{1} \int_{\R} |\tilde{g}(\theta,\rho)|^{2} \, d\rho \, d\theta, \end{equation}
where $\tilde{g}(\theta,\rho)$ is the Fourier transform in the $2^{nd}$ variable defined in \eqref{form53}. The formula \eqref{form54} follows by noting that $(\mathcal{F}g)(n,\rho)$ is the $n^{th}$ Fourier coefficient of the function $\theta \mapsto \tilde{g}(\theta,\rho)$, which is $1$-periodic by the hypothesis $g(0,\cdot) = g(1,\cdot)$. \end{remark}

\begin{remark} The Fourier transform $\mathcal{F}f$ is defined with respect to both the "angular" variable $\theta$ and the "radial" variable $r$. Therefore also the Sobolev regularity of $f$ encodes smoothness in both variables. However, if $f = Rg$ for some $g \in \mathcal{S}(\R^{2})$, formula \eqref{form2} shows that the angular and radial smoothness of $f$ are closely related. This observation is the key to the proof of the next lemma. The argument is virtually copied (up to notational conventions) from the proof of \cite[Theorem 5.2]{MR0856916}. \end{remark}

\begin{lemma}\label{lemma2} For every $\chi \in \mathcal{D}(\R^{2})$ there exists a constant $C_{\chi} > 0$ such that
\begin{displaymath} \|R(g\chi)\|_{H^{1}_{0}} \leq C_{\chi}\|g\|_{H^{1/2}_{0}}, \qquad g \in \mathcal{S}(\R^{2}). \end{displaymath} \end{lemma}

\begin{proof} Since $|(n,\rho)|^{2} \lesssim |n|^{2} + |\rho|^{2}$ for $(n,\rho) \in \Z \times \R$, we have
\begin{displaymath} \|R(g\chi)\|_{H^{1}_{0}}^{2} \lesssim \sum_{n \in \Z} \int_{\R} |\mathcal{F}[R(g\chi)](n,\rho) n|^{2} \, d\rho + \sum_{n \in \Z} \int_{\R} |\mathcal{F}[R(g\chi)](n,\rho)|^{2}|\rho|^{2} \, d\rho =: \Sigma_{1} + \Sigma_{2}. \end{displaymath}
To treat the term $I_{1}$, recall from basic Fourier series (or check by integration by parts) that if $h \in C^{1}([0,1])$ and $h(0) = h(1)$, then
\begin{displaymath} \widehat{\partial_{\theta}h}(n) = (2\pi i n) \hat{h}(n), \qquad n \in \Z. \end{displaymath}
Similarly for the Fourier transform: if $h \in C_{c}^{1}(\R)$, then $(2\pi i \rho)\hat{h}(\rho) = \widehat{\partial_{r}h}(\rho)$. For fixed $r \in \R$, the function $\theta \mapsto R(g\chi)(\theta,r)$ is $1$-periodic, so
\begin{align*} \mathcal{F}[R(g\chi)](n,\rho) n & = \tfrac{1}{2\pi i} \mathcal{F}(\partial_{\theta}[R(g\chi)])(n,\rho)\\
& \stackrel{\eqref{form2a}}{=} \tfrac{1}{2\pi i} \bar{e}_{\theta} \cdot \mathcal{F}(\partial_{r}[R(zg\chi)])(n,\rho) = \bar{e}_{\theta} \cdot \mathcal{F}[R(zg\chi)](n,\rho)\rho. \end{align*}
Abbreviate $\bar{g} := zg\chi \in \mathcal{S}(\R^{2},\R^{2})$.  Plugging the formula above into the definition of $\Sigma_{1}$, applying Plancherel in the first variable, and eventually integrating in polar coordinates, one finds
\begin{align}\label{form42} I_{1} \leq \int_{0}^{1}\int_{\R} |\widetilde{R\bar{g}}(\theta,\rho)|^{2}|\rho|^{2} \, d\rho \, d\theta & \stackrel{\eqref{form3}}{=} \int_{0}^{1} \int_{\R} |\hat{\bar{g}}(\rho e_{\theta})||\rho|^{2} \, d\rho \, d\theta\\
& \lesssim \int_{\R^{2}} |\hat{\bar{g}}(\xi)|^{2}|\xi| \, d\xi. \notag\end{align}
Since $g \mapsto z_{1}g\chi$ and $g \mapsto z_{2}g\chi$ are bounded operators on $H_{0}^{1/2}$, with norm depending on $\chi$ (see \cite[Chapter VII, Lemma 4.5]{MR0856916}), one finally deduces that 
\begin{equation}\label{form5a} \Sigma_{1} \lesssim_{\chi} \|g\|_{H^{1/2}_{0}}^{2}. \end{equation}
Treating the term $\Sigma_{2}$ is simpler, because the power $|\rho|^{2}$ already appears. In fact, we simply apply Plancherel in the first variable to deduce that
\begin{displaymath} \Sigma_{2} = \int_{0}^{1} \int_{\R} |\widetilde{R(g\chi)}(\theta,\rho)|^{2}|\rho|^{2} \, d\rho \, d\theta. \end{displaymath}
The continuation of the estimate is the same as on line \eqref{form42}, and thus $\Sigma_{2} \lesssim_{\chi} \|g\|_{H_{0}^{1/2}}^{2}$. \end{proof}

Lemma \ref{lemma2} provides one endpoint for an upcoming interpolation argument. The following lemma provides the second endpoint:

\begin{lemma}\label{lemma3} For every $g \in \mathcal{S}(\R^{2})$,
\begin{displaymath} \|Rg\|_{L^{2}([0,1] \times \R)} \lesssim \|g\|_{H^{-1/2}_{0}}. \end{displaymath} \end{lemma}

\begin{proof} Fix $g \in \mathcal{S}(\R^{2})$. In the following we first apply Plancherel in the $2^{nd}$ coordinate and then integrate in polar coordinates:
\begin{displaymath} \|Rg\|_{L^{2}}^{2} = \int_{0}^{1} \int_{\R} |\widetilde{Rg}(\theta,\rho)|^{2} \, d\rho \, d\theta \stackrel{\eqref{form3}}{=} \int_{0}^{1} \int_{\R} |\hat{g}(\rho e_{\theta})|^{2} \, d\rho \, d\theta \sim \int_{\R^{2}} |\hat{g}(\xi)|^{2}|\xi|^{-1} \, d\xi = \|g\|_{H^{-1/2}}^{2}. \end{displaymath}
This is what was claimed. \end{proof}

Lemmas \ref{lemma2}-\ref{lemma3} show (up to the $\chi$-multiplication) that $R$ maps 
\begin{displaymath} H_{0}^{1/2} \to H_{0}^{1} \quad \text{and} \quad H_{0}^{-1/2} \to L^{2} = H_{0}^{0}. \end{displaymath}
Now, for $s_{1} \neq s_{2}$ and $0 < \theta < 1$, one has
\begin{displaymath} H^{(1 - \theta) s_{1} + \theta s_{2}} = (H^{s_{1}},H^{s_{2}})_{[\theta]}. \end{displaymath}
That is, $H^{(1 - \theta) s_{1} + \theta s_{2}}$ is the complex interpolation space between $H_{0}^{s_{1}}$ and $H_{0}^{s_{2}}$ with parameter $\theta$. Indeed, the version of this for non-homogeneous Sobolev spaces is stated in \cite[Theorem 6.4.5(7)]{MR0482275}, and on \cite[p. 150]{MR0482275} the authors remark that the results in \cite[Chapter 6]{MR0482275} extend to homogeneous Sobolev spaces (see also \cite[Remark 6.9.3]{MR0482275}). Using this fact formally leads to the conclusion that $R$ maps 
\begin{displaymath} H^{s}_{0} \to H_{0}^{s + 1/2} \end{displaymath}
for $s \in [-1/2,1/2]$ (taking the multiplication by $\chi$ into account appropriately).

The minor technical issue with this reasoning is that our Sobolev space consists of functions defined on $[0,1] \times \R$, and not on $\R^{n}$ as in the reference \cite{MR0482275}, so the results of \cite{MR0482275} are not formally applicable. Instead of writing that "everything works the same way" I decided to provide an argument which avoids direct reference to interpolation spaces.

I will instead apply the following $L^{p}$-interpolation theorem of Stein and Weiss (see \cite[p. 17, Exercise 12]{MR0482275} or \cite[Theorem 5.4.1]{MR0482275} or \cite{MR0092943} for the original reference):
\begin{thm}[Stein-Weiss]\label{swInterpolation} Let $0 < p \leq \infty$ and $0 < \theta < 1$. Let $\mu,\nu$ be positive measures on spaces $X,Y$. Let $w_{0},w_{1}$ be non-negative $\mu$-measurable weights, and let $v_{0},v_{1}$ be non-negative $\nu$ measurable weights. Write $w_{\theta} := w_{0}^{1 - \theta}w_{1}^{\theta}$ and $v_{\theta} := v_{0}^{1 - \theta}v_{1}^{\theta}$.

Assume that $T$ is a linear map defined on $L^{p}(X,w_{0}d\mu) \cup L^{p}(X,w_{1}d\mu)$ which maps
\begin{displaymath} T \colon L^{p}(X,w_{0} d\mu) \to L^{p}(Y,v_{0}d\nu) \quad \text{and} \quad T \colon L^{p}(X,w_{1}d\mu) \to L^{p}(Y,v_{1}d\nu) \end{displaymath}
with operator norms $M_{0} \geq 0$ and $M_{1} \geq 0$ respectively. Then $T$ extends to an operator 
\begin{displaymath} T \colon L^{p}(X,w_{\theta}d\mu) \to L^{p}(Y,v_{\theta}d\nu) \end{displaymath}
with norm $\leq M_{0}^{1 - \theta}M_{1}^{\theta}$.
\end{thm}

\begin{thm}\label{thm1} For every $\chi \in \mathcal{D}(\R^{2})$ there exists a constant $C_{\chi} > 0$ such that 
\begin{equation}\label{form6a} \|R(g\chi)\|_{H_{0}^{s + 1/2}} \leq C_{\chi}\|g\|_{H_{0}^{s}}, \qquad g \in \mathcal{S}(\R^{2}), \, -\tfrac{1}{2} \leq s \leq \tfrac{1}{2}. \end{equation} \end{thm}

\begin{proof} The plan is to apply Theorem \ref{swInterpolation} in the spaces $X = \R^{2}$ and $Y = \Z \times \R$ equipped with the measures $\mu = \mathcal{L}^{2}$ (Lebesgue measure) and $\nu = \mathcal{H}^{0} \times \mathcal{L}^{1}$. 

Consider the following linear map $T := T_{\chi}$ which takes Schwartz functions to functions defined on $Y = \Z \times \R$:
\begin{displaymath} (Tg)(n,\rho) := \mathcal{F}[R(\widecheck{g}\chi)](n,\rho), \qquad (n,\rho) \in \Z \times \R, \, g \in \mathcal{S}(\R^{2}). \end{displaymath}
Consider the following non-negative weights on $\R^{2}$ and $\Z \times \R$, respectively
\begin{displaymath} \begin{cases} w_{0}(\xi) := |\xi|^{-1}\\ w_{1}(\xi) := |\xi| \end{cases} \quad \text{and} \quad \begin{cases} v_{0}(n,\rho) \equiv 1, \\ v_{1}(n,\rho) := |(n,\rho)|^{2}. \end{cases} \end{displaymath}
It now follows from Lemmas \ref{lemma2}-\ref{lemma3} that $T$ is (more precisely: extends to) a bounded linear map $L^{2}(X,w_{0}d\mu) \to L^{2}(Y,v_{0}d\nu)$ and $L^{2}(X,w_{1}d\mu) \to L^{2}(Y,v_{1}d\nu)$:
\begin{displaymath} \|Tg\|_{L^{2}(v_{0}d\nu)}^{2} = \|R(\widecheck{g}\chi)\|_{L^{2}([0,1] \times \R)}^{2} \stackrel{\mathrm{L. \,} \ref{lemma3}}{\lesssim} \|\widecheck{g}\chi\|_{H_{0}^{-1/2}}^{2} \lesssim_{\chi} \|\widecheck{g}\|_{H^{-1/2}_{0}}^{2} = \|g\|_{L^{2}(w_{0}d\mu)}^{2} \end{displaymath}
and
\begin{displaymath} \|Tg\|_{L^{2}(v_{1}d\nu)}^{2} = \|R(\widecheck{g}\chi)\|_{H^{1}_{0}}^{2} \stackrel{\mathrm{L. \,} \ref{lemma2}}{\lesssim_{\chi}} \|\widecheck{g}\|_{H^{1/2}_{0}}^{2} = \|g\|_{L^{2}(w_{1}d\mu)}^{2}. \end{displaymath} 
Consequently, by Theorem \ref{swInterpolation}, for all $0 < \theta < 1$ and $g \in \mathcal{S}(\R^{2})$,
\begin{displaymath} \sum_{n \in \Z} \int_{\R} |\mathcal{F}(R(\widecheck{g}\chi)](n,\rho)|^{2}|(n,\rho)|^{2\theta} \, d\rho = \|Tg\|_{L^{2}(v_{\theta}d\nu)}^{2} \lesssim_{\chi} \|g\|_{L^{2}(w_{\theta}d\mu)}^{2} =\int_{\R^{2}} |g(\xi)|^{2}|\xi|^{2\theta - 1} \, d\xi. \end{displaymath}
When applied to $\hat{g} \in \mathcal{S}(\R^{2})$ in place of $g$, this inequality is equivalent to \eqref{form6a}. \end{proof}

\subsection{Estimating incidences}\label{s:incidences}

In this section, the smoothing estimates for the $X$-ray transform developed above are applied to study weighted incidences between points and $\delta$-tubes in $\R^{2}$. We start by introducing terminology.

\begin{definition}[$\delta$-tube]\label{def:tube} For $\delta > 0$, a \emph{$\delta$-tube} is the closed $\delta$-neighbourhood of a line in $\R^{2}$. Write $\ell_{\theta,r} := \pi_{\theta}^{-1}\{r\}$, for $\theta \in [0,1]$, $r \in \R$. Then, denote
\begin{displaymath} T_{\theta,r}(\delta) := \{z \in \R^{2} : \dist(z,\ell_{\theta,r}) \leq \delta\}. \end{displaymath}
Thus $T_{\theta,r}$ is the $\delta$-tube indexed by $(\theta,r) \in [0,1] \times \R$. \end{definition}

\begin{definition}[Weighted incidences] Let $\mu$ be a finite Borel measure on $B(1)$, and let $\nu$ be a finite Borel measure on $[0,1] \times \R$. For $\delta > 0$ and $q \in [0,1] \times \R$, let $T_{q}(\delta)$ be the $\delta$-tube from Definition \ref{def:tube}. Set
\begin{equation}\label{def:incidences} \mathcal{I}_{\delta}(\mu,\nu) := (\mu \times \nu)(\{(p,q) : p \in T_{q}(\delta)\}). \end{equation}
\end{definition}

The following lemma is key to relating incidences to the $X$-ray transform:

\begin{lemma}\label{lemma4} Let $q = (\theta_{0},r_{0}) \in (0,1) \times \R$ and $\delta \in (0,\theta_{0})$. Recall the $\delta$-tube $T_{q}(\delta)$ from Definition \ref{def:tube}. Then,
\begin{displaymath} \mathbf{1}_{T_{q}(\delta)}(p) \leq \delta^{-1} \int_{0}^{1} \mathbf{1}_{B(q,3\delta)}(\theta,\pi_{\theta}(p)) \, d\theta, \qquad p \in B(1). \end{displaymath}  \end{lemma}

\begin{proof} Fix $p \in B(1) \cap T_{q}(\delta)$. It suffices to show that  $(\theta,\pi_{\theta}(p)) \in B(q,3\delta)$ for all $|\theta - \theta_{0}| \leq \delta$. By definition $T_{q}(\delta) = \{z \in \R^{2} : \dist(z,\pi_{\theta_{0}}^{-1}\{r_{0}\}) \leq \delta\}$. Therefore $\dist(p,\pi_{\theta_{0}}^{-1}\{r_{0}\}) \leq \delta$, and consequently
\begin{equation}\label{form7} |\pi_{\theta_{0}}(p) - r_{0}| \leq \delta. \end{equation} 
If $|\theta - \theta_{0}| \leq \delta$, one has $|\pi_{\theta}(p) - \pi_{\theta_{0}}(p)| \leq |\theta - \theta_{0}| \leq \delta$ since $p \in B(1)$, and therefore
\begin{displaymath} |\pi_{\theta}(p) - r_{0}| \leq |\pi_{\theta}(p) - \pi_{\theta_{0}}(p)| + |\pi_{\theta_{0}}(p) - r_{0}| \stackrel{\eqref{form7}}{\leq} 2\delta. \end{displaymath}
This implies that $(\theta,\pi_{\theta}(p)) \in B(q,3\delta)$ for all $|\theta - \theta_{0}| \leq \delta$. \end{proof}

Now all is ready to prove Theorem \ref{main2}, which I restate below:

\begin{thm}\label{thm2a} Let $\mu$ be a finite Borel measure on $B(1)$, and let $\nu$ be a finite Borel measure on $[c,1 - c] \times [-1,1]$ for some $c \in (0,1)$. Let $\delta \in (0,c/7]$. Then,
\begin{displaymath} \mathcal{I}_{\delta}(\mu,\nu) \lesssim_{t} \delta \cdot \sqrt{I_{3 - t}(\mu)I_{t}(\nu)}, \qquad t \in (1,2). \end{displaymath} \end{thm}

\begin{proof} I only consider the case where $\mu \in C^{\infty}_{c}(\R^{2})$. The general case is easily reduced to this via convolution approximation.

Let $\psi \in C^{\infty}_{c}(\R^{2})$ be a function with $\mathbf{1}_{B(0,1/2)} \leq \psi \leq \mathbf{1}_{B(1)}$ and $\int \psi \sim 1$. As usual, write $\psi_{\delta}(q) := \delta^{-2}\psi_{\delta}(q/\delta)$. For $\delta > 0$, consider $\nu_{\delta} := \nu \ast \psi_{6\delta} \in C^{\infty}_{c}(\R^{2})$, and notice that
\begin{equation}\label{form8} \nu(B(q,3\delta)) \lesssim \delta^{2} \nu_{\delta}(q), \qquad q \in \R^{2}. \end{equation}
Next, fix $p \in \spt \mu \subset B(1)$. Applying Lemma \ref{lemma4} for each $q \in \spt \nu \subset [c,1 - c] \times \R$, and also using Fubini's theorem,
\begin{align*} \int \mathbf{1}_{T_{q}(\delta)}(p) \, d\nu q & \leq \delta^{-1} \int_{0}^{1} \mathbf{1}_{B(q,3\delta)}(\theta,\pi_{\theta}(p)) \, d\theta \, d\nu q\\
& = \delta^{-1} \int_{0}^{1} \nu((\theta,\pi_{\theta}(p)),3\delta) \, d\theta \stackrel{\eqref{form8}}{\lesssim} \delta \int_{0}^{1} \nu_{\delta}(\theta,\pi_{\theta}(p)) \, d\theta = \delta R^{\ast}(\nu_{\delta})(p). \end{align*} 
Consequently, by the definition of $\mathcal{I}_{\delta}(\mu,\nu)$,
\begin{displaymath} \delta^{-1} \mathcal{I}_{\delta}(\mu,\nu) = \delta^{-1} \iint \mathbf{1}_{T_{q}(\delta)}(p) \, d\nu q \, d\mu p \lesssim \int R^{\ast}(\nu_{\delta})(p) \, d\mu p. \end{displaymath}
So far the continuous density of $\mu$ was not needed, but it is used next to make sense of $R\mu$, and to apply the adjoint property of $R,R^{\ast}$ legitimately:
\begin{equation}\label{form52} \delta^{-1} \mathcal{I}_{\delta}(\mu,\nu)  \lesssim \int_{0}^{1} \int_{\R} \nu_{\delta}(\theta,r) (R\mu)(\theta,r) \, dr \, d\theta. \end{equation}
Since $\nu$ is compactly supported on $[c,1 - c] \times \R$, and $\delta \leq c/7$, the convolution $\nu_{\delta} = \nu \ast \psi_{6\delta}$ is still compactly supported on $[0,1] \times \R$. In particular both $\nu_{\delta}$ and $R\mu$ are $1$-periodic (or more precisely have $1$-periodic extensions) in the $\theta$-variable. Motivated by this, one uses Plancherel and Cauchy-Schwarz, at the same time introducing the parameter $t \in (1,2)$:
\begin{align*} \delta^{-1} \mathcal{I}_{\delta}(\mu,\nu) & \lesssim \sum_{n \in \N} \int_{\R} |\mathcal{F}(\nu_{\delta})(n,\rho)||\mathcal{F}(R\mu)(n,\rho)| \, d\rho\\
& \leq \Big(\sum_{n \in \Z} \int_{\R} |\mathcal{F}(\nu_{\delta})(n,\rho)|^{2}|(n,\rho)|^{t - 2} \, d\rho \Big)^{1/2} \Big(\sum_{n \in \Z} \int_{\R} |\mathcal{F}(R\mu)(n,\rho)|^{2}|(n,\rho)|^{2 - t} \, d\rho \Big)^{1/2}. \end{align*} 
Denote the two factors above by $\Pi_{1}$ and $\Pi_{2}$. The claim is that
\begin{displaymath} \Pi_{1} \lesssim_{t} I_{t}(\nu)^{1/2} \quad \text{and} \quad \Pi_{2} \lesssim_{t} I_{3 - t}(\mu)^{1/2}. \end{displaymath} 

Consider first $\Pi_{1}$. Start by noting that 
\begin{displaymath} \mathcal{F}(\nu_{\delta})(n,\rho) = \int_{0}^{1} \int_{\R} e^{-2\pi i(n\theta + \rho r)} \nu_{\delta}(\theta,r) \, d\rho \, d\theta = \widehat{\nu_{\delta}}(n,\rho) \end{displaymath}
coincides with the usual Fourier transform of $\nu_{\delta}$ in $\R^{2}$, since $\spt \nu_{\delta} \subset [0,1] \times \R$. In particular 
\begin{displaymath} |\mathcal{F}(\nu_{\delta})(n,\rho)| \lesssim |\hat{\nu}(n,\rho)|, \qquad \delta > 0, \, (n,\rho) \in \Z \times \R. \end{displaymath}
Moreover, since $\nu$ is supported on $[0,1] \times \R$, one may easily check (employing the usual trick of writing $d\nu(\theta,r) = \chi(\theta)d\nu(\theta,r)$ for a suitable $\chi \in \mathcal{S}(\R)$) that
\begin{displaymath} |\hat{\nu}(n,\rho)| \lesssim \int_{n - 1}^{n + 1} |\hat{\nu}(\eta,\rho)| \, d\eta, \qquad (n,\rho) \in \Z \times \R. \end{displaymath}
Consequently,
\begin{equation}\label{form43} \sum_{n \in \Z} \int_{\R} |\mathcal{F}(\nu_{\delta})(n,\rho)|^{2}|(n,\rho)|^{t - 2} \, d\rho \lesssim \sum_{n \in \Z} \int_{\R} \Big( \int_{n - 1}^{n + 1} |\hat{\nu}(\eta,\rho)|^{2} \, d\eta \Big) |(n,\rho)|^{t - 2} \, d\rho. \end{equation}
If either $|n| \geq 1$ or $n = 0$ and $|\rho| \geq 1$, one has 
\begin{displaymath} |(n,\rho)|^{t - 2} \lesssim \inf \{|(\eta,\rho)|^{t - 2} : n - 1 \leq \eta \leq n + 1\}. \end{displaymath}
Therefore, isolating the pairs $\{(0,\rho) : |\rho| \leq 1\}$, the right hand side of \eqref{form43} can be further estimated as
\begin{displaymath} \int_{-1}^{1} \int_{-1}^{1} |\hat{\nu}(\eta,\rho)|^{2}|\rho|^{t - 2} \, d\eta \, d\rho + \int_{\R^{2}} |\hat{\nu}(\xi)|^{2}|\xi|^{t - 2} \, d\xi. \end{displaymath}
The second term is $\sim_{t} I_{t}(\mu)$ by  \cite[Lemma 12.12]{zbMATH01249699}. For the first term, simply estimate $|\hat{\nu}(\eta,\rho)|^{2} \leq \nu(\R^{2})^{2} \lesssim_{t} I_{t}(\nu)$, recalling that $\spt \nu \subset [-1,1]^{2}$, and then note that $\int_{-1}^{1} |\rho|^{t - 2} \, d\rho \lesssim_{t} 1$ since $t > 1$. Putting these estimates together shows that $\Pi_{1} \lesssim_{t} I_{t}(\nu)^{1/2}$.

Moving on to $\Pi_{2}$, one may resort to Theorem \ref{thm1}, after realising that
\begin{displaymath} \Pi_{2}^{2} = \sum_{n \in \Z} \int_{\R} |\mathcal{F}(R\mu)(n,\rho)|^{2}|(n,\rho)|^{2 - t} \, d\rho = \|R\mu\|_{H_{0}^{(2 - t)/2}}^{2}. \end{displaymath}
Since $(2 - t)/2 = s + 1/2$ for $s = (1 - t)/2 \in (-\tfrac{1}{2},0)$ and $t \in (1,2)$, and $\mu \in C^{\infty}_{c}(B(1))$, one may deduce from \eqref{form6a} that
\begin{displaymath} \|R\mu\|_{H_{0}^{(2 - t)/2}}^{2} \lesssim \|\mu\|_{H_{0}^{(1 - t)/2}}^{2} = \int_{\R^{2}} |\hat{\mu}(\xi)|^{2}|\xi|^{1 - t} \, d\xi = \int_{\R^{2}} |\hat{\mu}(\xi)|^{2}|\xi|^{(3 - t) - 2} \, d\xi \sim_{t} I_{3 - t}(\mu). \end{displaymath}
In the final estimate \cite[Lemma 12.12]{zbMATH01249699} was applied again. Combining the estimates above completes the proof of $\Pi_{2} \lesssim_{t} I_{3 - t}(\mu)^{1/2}$, and the proof of Theorem \ref{thm2a}.  \end{proof}

\section{Discretising the main theorem}\label{s:discretisation}

I start moving towards the proof of Theorem \ref{thm2}. The first task it to reduce Theorem \ref{thm2} to a suitable $\delta$-discretised counterpart. This is the statement below:

\begin{thm}\label{thm3} For every $s \in (0,1]$ and $t \in (1,2]$ such that $s + t > 2$, for every $\tau \in (1,t)$, and $C \geq 1$, there exists a constants $\delta_{0} = \delta_{0}(C,s,t,\tau) > 0$ and $\epsilon = \epsilon(C,s,t,\tau) > 0$ such that the following holds for all $\delta \in 2^{-\N} \cap (0,\delta_{0}]$.

Let $\mu,\nu$ be Borel probability measures supported on $B(1)$, and satisfying $\dist(\spt \mu,\spt \nu) \geq \tfrac{1}{2}$. Assume that
\begin{displaymath} \mu(B(x,r)) \leq Cr^{t} \quad \text{and} \quad \nu(B(x,r)) \leq Cr^{s}, \qquad x \in \R^{2}, \, r > 0. \end{displaymath}
Let $E \subset \spt \nu$ and $F \subset \spt \mu$ be Borel sets with full $\nu$ measure and $\mu$ measure, respectively. Assume that to every $x \in E$ there corresponds a family $\mathcal{T}_{x}$ of $\delta$-tubes such that $x \in T$ for all $T \in \mathcal{T}_{x}$, and
\begin{equation}\label{form13b} \mu(\cup \mathcal{T}_{x}) \geq C^{-1}, \qquad x \in E. \end{equation}
Then, there exists $x \in E$ and a tube $T \in \mathcal{T}_{x}$ such that $\mathcal{H}^{\tau - 1}_{\delta,\infty}(F \cap 10T) \geq \epsilon$. Here $10T = \ell_{\theta,r}(10\delta)$ if $T = \ell_{\theta,r}(\delta)$. \end{thm}

Here $\mathcal{H}^{\tau - 1}_{\delta,\infty}$ refers to a variant of the usual Hausdorff content where the covering sets are balls of radii in $[\delta,\infty)$. The purpose of the remainder of this section is to reduce Theorem \ref{thm2} to Theorem \ref{thm3}:

\begin{proof}[Proof of Theorem \ref{thm2} assuming Theorem \ref{thm3}] Fix $t \in (1,2)$ and a Borel set $B \subset \R^{2}$ as in Theorem \ref{thm2}. One may assume that $B \subset \R^{2}$ is compact; this reduction is simply achieved by choosing a compact subset $F \subset B$ with $\Hd F > t$. From now on, I will use the notation "$F$" in place of "$B$". 

Given a compact set $F \subset \R^{2}$ with $\Hd F > t$, it suffices to prove the following: if $E \subset \R^{2} \, \setminus \, F$ is another compact set with $\Hd E > 2 - t$, then there exists a point $x \in E$ such that
\begin{equation}\label{form10} \mathcal{H}^{1}(\{e \in S^{1} : \Hd (F \cap \ell_{x,e}) \geq t - 1\}) > 0. \end{equation}
This is what we aim to prove. However, it is useful to observe the one more reduction. The claim above is formally equivalent to a superficially weaker version of itself, where the conclusion \eqref{form10} is relaxed to: \emph{for every $\tau \in (1,t)$, there exists $x \in E$ such that}
\begin{equation}\label{form11} \mathcal{H}^{1}(\{e \in S^{1} : \mathcal{H}^{\tau - 1}(F \cap \ell_{x,e}) > 0\}) > 0. \end{equation}
Indeed, if this version has already been proven, and we are provided compact sets $E,F$ with $\Hd E > 2 - t$ and $\Hd F > t$, we observe that $\Hd E > 2 - \bar{t}$ and $\Hd F > \bar{t}$ for some $\bar{t} \in (t,2)$. At this point \eqref{form11} may be applied with $\tau := t \in (1,\bar{t})$ to deduce \eqref{form10}.

We finally prove the weaker version \eqref{form11}. Fix compact sets $E,F \subset \R$ with $\Hd E > 2 - t$ and $\Hd F > t$. We assume with no loss of generality that $E,F \subset B(1)$ with $\dist(E,F) \geq \tfrac{1}{2}$. We make a counter assumption: there exists $\tau \in (1,t)$ such that
\begin{equation}\label{form12} \mathcal{H}^{\tau - 1}(F \cap \ell_{x,e}) = 0 \end{equation}
for all $x \in E$, for $\mathcal{H}^{1}$ a.e. $e \in S^{1}$. 

Let $\mu,\nu$ be Borel probability measures on $F,E$, respectively, satisfying $\mu(B(x,r)) \leq Cr^{t}$ and $\nu(B(x,r)) \leq Cr^{s}$ for all $x \in \R^{2}$ and $r > 0$, where $s + t > 2$. By renaming the sets $E,F$ if necessary, we assume that
\begin{displaymath} F = \spt \mu \quad \text{and} \quad E = \spt \nu. \end{displaymath}
It follows from \cite[Theorem 1.11]{MR3892404} that there exists a set $E_{0} \subset E$ of full $\nu$ measure such that
\begin{displaymath} \pi_{x}\mu \ll \mathcal{H}^{1}|_{S^{1}}, \qquad x \in E_{0}. \end{displaymath}
Fix $x \in E_{0}$. Since $\pi_{x}\mu \ll \mathcal{H}^{1}|_{S^{1}}$, our counter assumption \eqref{form12} implies:
\begin{displaymath} \mathcal{H}^{\tau - 1}(F \cap \ell_{x,e}) = 0 \quad \text{ for $\pi_{x}\mu$ a.e. $e \in S^{1}$}. \end{displaymath}
In other words, the set $S_{x}' := \{e \in S^{1} : \mathcal{H}^{\tau - 1}(F \cap \ell_{x,e}) = 0\}$ has $(\pi_{x}\mu)(S_{x}') = 1$. 

Write
\begin{equation}\label{form14} \epsilon := \epsilon(2C,s,t,\tau) > 0, \end{equation}
where $\epsilon(2C,s,t,\tau) > 0$ is the constant provided by Theorem \ref{thm3}.

For every $x \in E_{0}$ and $e \in S_{x}'$, choose a cover $\mathcal{B}_{x,e}$ of $F \cap \ell_{x,e}$ by open discs satisfying
\begin{equation}\label{form9} \sum_{B \in \mathcal{B}_{x,e}} \diam(B)^{\tau - 1} < \epsilon. \end{equation}
Since $F \cap \ell_{x,e}$ is compact, we may assume that $|\mathcal{B}_{x,e}| < \infty$.

Fix $x \in E_{0}$, $e \in S_{x}'$, and write $B_{x,e} := \bigcup \{B : B \in \mathcal{B}_{x,e}\}$. We claim that there exists $\delta_{x,e} > 0$ such that
\begin{equation}\label{form6} F \cap \ell_{x,e}(\delta_{x,e}) \subset B_{x,e}. \end{equation}
If this were not the case, we might choose radii $\delta_{j} \searrow 0$ and points 
\begin{displaymath} y_{j} \in (F \cap \ell_{x,e}(\delta_{j})) \, \setminus \, B_{x,e} \subset F. \end{displaymath}
Since $F$ is compact, the points $y_{j}$ converge (up to a subsequence) to a limit $y \in F \cap \ell_{x,e}$. On the other hand, $y_{j} \in \R^{2} \, \setminus \, B_{x,e}$ for all $j \in \N$. Since $B_{x,e}$ is open, also the limit $y \in (F \cap \ell_{x,e}) \, \setminus \, B_{x,e}$. This is a contradiction, since $F \cap \ell_{x,e} \subset B_{x,e}$. 

Note that since $\mathcal{B}_{x,e}$ is finite, and $\delta_{x,e}$ may be taken smaller if desired, we may assume that 
\begin{equation}\label{form7} 0 < \delta_{x,e} \leq \min\{\diam(B) : B \in \mathcal{B}_{x,e}\}. \end{equation}
The scale $\delta_{x,e} > 0$ depends on $x$ and $e$, but we next seek to remove the dependence by passing to suitable subsets. Namely, first for every $x \in E_{0}$, let $\delta_{x} > 0$ be so small that $\delta_{x,e} \geq \delta_{x}$ for a set $S_{x} \subset S_{x}'$ with $(\pi_{x}\mu)(S_{x}) \geq \tfrac{1}{2}$. Finally, let $\delta \in 2^{-\N}$ be so small that $\delta_{x} \geq 10\delta$ for a Borel set $E_{1} \subset E_{0}$ with $\nu(E_{1}) \geq \tfrac{1}{2}$. In summary, 
\begin{equation}\label{form15} 0 < 10\delta \leq \delta_{x} \leq \delta_{x,e}, \qquad x \in E_{1}, \, e \in S_{x}. \end{equation}
Let $\bar{\nu} := \nu(E_{1})^{-1}\nu|_{E_{1}}$. Then $\bar{\nu}$ is a Borel probability measure, satisfies the Frostman condition $\bar{\nu}(B(x,r)) \leq 2Cr^{s}$, and $E_{1} \subset \spt \bar{\nu}$ has full $\bar{\nu}$ measure.

For $x \in E_{1}$, define the family $\mathcal{T}_{x}$ of $\delta$-tubes $\mathcal{T}_{x} := \{\ell_{x,e}(\delta) : e \in S_{x}\}$, where $\ell_{x,e} = x + \mathrm{span}(e)$ is the line in direction $e$ containing $x$. Then
\begin{displaymath} \mu\left(\cup \mathcal{T}_{x} \right) \geq (\pi_{x}\mu)(S_{x}) \geq \tfrac{1}{2}, \qquad x \in E_{1}, \end{displaymath} 
This verifies property \eqref{form13b}, in fact with absolute constant "$\tfrac{1}{2}$".

Finally, let us check that
\begin{equation}\label{form8} \mathcal{H}^{\tau - 1}_{\delta,\infty}(F \cap 10T) < \epsilon \stackrel{\eqref{form14}}{=} \epsilon(2C,s,t,\tau), \qquad x \in E_{1}, \, T \in \mathcal{T}_{x}. \end{equation}
This will contradict Theorem \ref{thm3} applied to $\mu,\bar{\nu}$, and therefore complete the proof of Theorem \ref{thm2}. To see \eqref{form8}, fix $x \in E_{1}$ and $T = \ell_{x,e}(\delta) \in \mathcal{T}_{p}$, where $e \in S_{x}$. Recall from \eqref{form6}+\eqref{form15} that 
\begin{displaymath} F \cap 10T = F \cap \ell_{x,e}(10\delta) \subset F \cap \ell_{x,e}(\delta_{x,e}) \subset B_{x,e}. \end{displaymath}
Now \eqref{form8} follows immediately from \eqref{form9}, and recalling that $\diam(B) \geq \delta_{x,e} \geq \delta$ for all $B \in \mathcal{B}_{x,e}$ with $x \in E_{1}$ and $e \in S_{x}$. \end{proof}

\section{A quantitative Furstenberg set estimate}

Starting from this section, it will be convenient to use \emph{dyadic $\delta$-tubes}. These objects have appeared many times in approximate incidence geometry literature, see for example \cite[Section 2.3]{OS23} for a thorough treatment. I use a slightly different notational convention, introduced below (this is simply because lines in this paper are parametrised as $\ell_{\theta,r} = \pi_{\theta}^{-1}\{r\}$, and not as $\ell_{a,b} = \{(x,y) : y = ax + b\}$, as in \cite{OS23}).
\begin{definition}[Dyadic $\delta$-tubes] Let $\delta \in 2^{-\N}$, and let $\mathcal{D}_{\delta} = \mathcal{D}_{\delta}([0,1) \times \R)$ be the standard dyadic $\delta$-squares contained in $[0,1) \times \R$. For $(\theta,r) \in [0,1) \times \R$, recall the notation $\ell_{\theta,r} = \pi_{\theta}^{-1}\{r\}$, where $\pi_{\theta}(z) = z \cdot (\cos 2\pi \theta,\sin 2\pi \theta)$.

For $Q \in \mathcal{D}_{\delta}$, define the set
\begin{displaymath} T(Q) := \bigcup \{\ell_{\theta,r} : (\theta,r) \in Q\}. \end{displaymath}
The collection $\mathcal{T}^{\delta} := \{T(Q) : Q \in \mathcal{D}_{\delta}\}$ is called the family of \emph{dyadic $\delta$-tubes}. If $\mathcal{T} \subset \mathcal{T}^{\delta}$ is a family of dyadic $\delta$-tubes, the notation $\mathcal{H}^{s}_{\delta,\infty}(\mathcal{T})$ refers to the Hausdorff content of the family of parameter squares $\{Q \in \mathcal{D}_{\delta} : T(Q) \in \mathcal{T}\}$. \end{definition} 

\begin{remark} Since elements of $\mathcal{T}^{\delta}$ contain lines of various slopes, they do not resemble ordinary $\delta$-tubes at scales much greater than $1$. However, all of our analysis of dyadic $\delta$-tubes happens inside $B(1)$, and there the geometry of dyadic $\delta$-tubes is roughly the same as that of ordinary $\delta$-tubes. For example, I record that if $Q \in \mathcal{D}_{\delta}$, then
\begin{equation}\label{form40} T(Q) \cap B(2) \subset \ell_{\theta,r}(10\delta), \qquad (\theta,r) \in Q. \end{equation}
\end{remark}

The statement of Theorem \ref{thm3a} below, and its proof, uses the notions of both \emph{$(\delta,s)$-sets} and \emph{Katz-Tao $(\delta,s)$-sets}. These are defined below.

\begin{definition}[Katz-Tao $(\delta,s,C)$-set]\label{def:KT} Given a metric space $(X,d)$, and a constant $C > 0$, a \emph{Katz-Tao $(\delta,s,C)$-set} is a set $P \subset X$ satisfying
\begin{equation}\label{form56} |P \cap B(x,r)|_{\delta} \leq C\left(\tfrac{r}{\delta}\right)^{s}, \qquad x \in X, \, r \geq \delta. \end{equation}
Here $|\cdot|_{\delta}$ is the $\delta$-covering number. If the constant $C > 0$ is absolute, $P$ is called simply a \emph{Katz-Tao $(\delta,s)$-set}. A family of dyadic $\delta$-tubes $\mathcal{T}$ is called a Katz-Tao $(\delta,s,C)$-set if the union of the parameter squares $\{Q : T(Q) \in \mathcal{T}\}$ forms a Katz-Tao $(\delta,s,C)$-subset of $\R^{2}$. \end{definition}

Definition \ref{def:KT} is due to Katz and Tao \cite[Definition 1.2]{KT01}. It is distinct from the following notion of \emph{$(\delta,s)$-sets}:
\begin{definition}[$(\delta,s,C)$-set]\label{def:deltaSSet} With the notation of Definition \ref{def:KT}, the set $P \subset X$ is called a \emph{$(\delta,s,C)$-set} if
\begin{equation}\label{form55} |P \cap B(x,r)|_{\delta} \leq Cr^{s}|P|_{\delta}, \quad x \in X, \, r \geq \delta. \end{equation} \end{definition}

\begin{remark} To grasp the difference between $(\delta,s)$-sets and Katz-Tao $(\delta,s)$-sets, the following observation is useful: a non-empty $(\delta,s,C)$-set $P \subset X$ satisfies $|P|_{\delta} \geq \delta^{-s}/C$ (apply \eqref{form55} with $r = \delta$), whereas a Katz-Tao $(\delta,s,C)$-set $P \subset B(1)$ satisfies $|P|_{\delta} \leq C\delta^{-s}$ (apply \eqref{form56} with $r = 1$). 

In fact, every $(\delta,s,C)$-set can be written as a disjoint union of $\lessapprox C|P|\delta^{t}$ Katz-Tao $(\delta,s)$-sets, see \cite[Lemma 3.2]{2022arXiv220513890O}. All statements below could in principle be expressed in terms of the Katz-Tao condition only, but using both definitions leads to a cleaner exposition.  \end{remark}

Theorem \ref{thm3a} is a quantitative version of the case $s + t \geq 2$ of \cite[Theorem 1.6]{fu2022incidence}. The proof is rather different from the argument presented in \cite[Section 5]{fu2022incidence}.

\begin{thm}\label{thm3a} Let $t \in (1,2]$, $s \in (2 - t,1]$, $C \geq 1$. Assume that $\mu$ is Borel measure on $B(1) \subset \R^{2}$ satisfying $\mu(B(p,r)) \leq Cr^{t}$ for all $p \in \R^{2}$ and $r > 0$. Fix $\delta \in 2^{-\N}$, and for every $p \in \mathcal{D}_{\delta}(\spt \mu)$, let $\mathcal{T}_{p} \subset \mathcal{T}^{\delta}$ be a $(\delta,s,C)$-set such that $T \cap p \neq \emptyset$ for all $T \in \mathcal{T}_{p}$. Then,
\begin{equation}\label{form59} \mathcal{H}_{\delta,\infty}^{\sigma + 1}(\mathcal{T}) \gtrsim_{\sigma,s,t} \mu(\R^{2})/C^{3}, \qquad \sigma \in [0,s). \end{equation}
In particular, $\mathcal{T}$ contains a Katz-Tao $(\delta,\sigma + 1)$-set $\mathcal{T}'$ of cardinality $|\mathcal{T}'| \gtrsim_{\sigma,s,t} \mu(\R^{2})\delta^{-(\sigma + 1)}/C^{3}$. \end{thm}

\begin{remark} The "In particular..." part follows immediately from the Hausdorff content lower bound combined with \cite[Proposition A.1]{FaO}. \end{remark}

\begin{proof}[Proof of Theorem \ref{thm3a}] One may assume that $t \in (1,2)$, because if $t = 2$, the measure $\mu$ also satisfies $\mu(B(p,r)) \leq Cr^{t'}$ for some $t' \in (1,2)$ with $s \in (2 - t',1]$.

Fix $\sigma \in (2 - t,s)$ (the cases $\sigma \in [0,2 - t]$ of \eqref{form59} follow by monotonicity if one manages to treat $\sigma \in (2 - t,s)$.)  Assume to the contrary that
\begin{equation}\label{form13a} \mathcal{H}_{\delta,\infty}^{\sigma + 1}(\mathcal{T}) \leq \epsilon \cdot \mu(\R^{2})/C^{3}, \end{equation}
where
\begin{equation}\label{form12a} \epsilon := \epsilon(\sigma,s,t) := \inf_{0 < \Delta \leq 1/2} \frac{c\Delta^{2(\sigma - s)}}{\log^{10}(1/\Delta)} > 0, \end{equation}
and $c = c(\sigma,t) > 0$ is a constant to be determined later.

The counter assumption \eqref{form13a} can be used to effectively cover $\mathcal{T}$ by Katz-Tao $(\sigma + 1,\Delta_{j})$-sets at various scales $\Delta_{j} \in 2^{-\N} \cap [\delta,\tfrac{1}{2}]$. The existence of such coverings was originally observed by Katz and Tao \cite[Lemma 7.5(a)]{KT01}, but refined versions have later appeared in \cite[Lemma 2]{2022arXiv220713844G} and \cite[Lemma 11]{2022arXiv220803597B}. I literally use the version from \cite{2022arXiv220803597B} (whose proof is a striking $3$ lines long). Namely, for every $\Delta_{j} = 2^{-j} \in 2^{-\N} \cap [\delta,\tfrac{1}{2}]$, one may find a (possibly empty) Katz-Tao $(\sigma + 1,\Delta_{j})$-set $\mathbb{T}_{j} \subset \mathcal{T}^{\Delta_{j}}$ such that:
\begin{itemize}
\item[(a) \phantomsection \label{a}] $\sum_{j = 1}^{\log(1/\delta)} \Delta_{j}^{\sigma + 1}|\mathbb{T}_{j}| \lesssim \mathcal{H}^{\sigma + 1}_{\delta,\infty}(\mathcal{T}) \leq \epsilon \cdot \mu(\R^{2})/C^{3}$.
\item[(b) \phantomsection \label{b}] every tube $T \in \mathcal{T}$ is contained in exactly one tube from some collection $\mathbb{T}_{j}$.  
\end{itemize}

We then use the pigeonhole principle to find a particularly useful index "$j$". Start by noting that for every $p \in \mathcal{P}$ fixed, property \nref{b} implies
\begin{displaymath} \mathcal{T}_{p} = \bigcup_{j = 1}^{\log(1/\delta)} (\mathcal{T}_{p} \cap \mathbb{T}_{j}). \end{displaymath}
(Here $\mathcal{T}_{p} \cap \mathbb{T}_{j}$ consists of those tubes in $\mathcal{T}_{p}$ whose dyadic $\Delta_{j}$-ancestor lies in $\mathbb{T}_{j}$.) Therefore, one may find an index $j = j(p) \in \{1,\ldots,\log(1/\delta)\}$ such that 
\begin{equation}\label{form10a} |\mathcal{T}_{p} \cap \mathbb{T}_{j}| \gtrsim |\mathcal{T}_{p}|/j^{2}. \end{equation}
Write $\mathcal{P}_{j} := \{p \in \mathcal{P} : j(p) = j\}$, for $j \in \{1,\ldots,\log(1/\delta)\}$. Then $\mathcal{P}$ is contained in the union of the sets $\mathcal{P}_{j}$, so one may find and fix an index $j \in \{1,\ldots,\log(1/\delta)\}$ such that $\mu(\cup \mathcal{P}_{j}) \gtrsim \mu(\R^{2})/j^{2}$. We write $\Delta := 2^{-j}$ for the index "$j$" found above.

Write $\mathcal{P}_{\Delta} := \mathcal{D}_{\Delta}(\mathcal{P}_{j})$, and let $\mu_{\Delta}$ be the absolutely continuous measure on $\R^{2}$ which has constant density of $\Delta$-squares, and is determined by
\begin{displaymath} \mu_{\Delta}(\mathbf{p}) := \mu(\mathbf{p}), \qquad p \in \mathcal{P}_{\Delta}. \end{displaymath}
Note that 
\begin{equation}\label{form58} \mu_{\Delta}(\cup \mathcal{P}_{\Delta}) \geq \mu(\cup \mathcal{P}_{j}) \gtrsim \mu(\R^{2})/\log^{2}(1/\Delta). \end{equation}
Furthermore, it follows easily from the hypothesis $\mu(B(x,r)) \leq Cr^{t}$ that 
\begin{equation}\label{form11a} I_{t}(\mu_{\Delta}) \lesssim_{t} C\mu_{\Delta}(\R^{2})\log(1/\Delta) \leq C\mu(\R^{2})\log(1/\Delta). \end{equation}

Abbreviate $\mathbb{T} := \mathbb{T}_{j}$. For $\mathbf{p} \in \mathcal{P}_{\Delta}$, define
\begin{displaymath} \mathbb{T}_{\mathbf{p}} := \{\mathbf{T} \in \mathbb{T} : \mathbf{T} \cap \mathbf{p} \neq \emptyset\}. \end{displaymath}
Let us check that
\begin{equation}\label{form9a} |\mathbb{T}_{\mathbf{p}}| \gtrsim \frac{\Delta^{-s}}{C\log^{2}(1/\Delta)},  \qquad \mathbf{p} \in \mathcal{P}_{\Delta}. \end{equation}
If $\mathbf{p} \in \mathcal{P}_{\Delta}$, one may pick $p \in \mathcal{P}_{j}$ with $p \subset \mathbf{p}$. First observe that $\mathbb{T}_{\mathbf{p}}$ contains all the dyadic $\Delta$-ancestors of $\mathcal{T}_{p} \cap \mathbb{T}$. Namely, if $T \in \mathcal{T}_{p} \cap \mathbb{T}$, then (a) the dyadic $\Delta$-ancestor $\mathbf{T}$ of $T$ is contained in $\mathbb{T}$ by definition of $\mathcal{T}_{p} \cap \mathbb{T}$, and (b) $T \cap p \neq \emptyset$ by definition of $\mathcal{T}_{p}$. Therefore also $\mathbf{T} \cap \mathbf{p} \supset T \cap p \neq \emptyset$, hence $\mathbf{T} \in \mathbb{T}_{\mathbf{p}}$.

Next, by the $(\delta,s,C)$-set property of $\mathcal{T}_{p}$, every fixed ancestor $\mathbf{T} \in \mathbb{T}_{\mathbf{p}}$ can contain $\leq C\Delta^{s}|\mathcal{T}_{p}|$ elements of $\mathcal{T}_{p} \cap \mathbb{T}$. Combining this observation with \eqref{form10a} leads to the inequality
\begin{displaymath} |\mathcal{T}_{p}|/\log^{2}(1/\Delta) \lesssim |\mathcal{T}_{p} \cap \mathbb{T}| \leq |\mathbb{T}_{\mathbf{p}}| \cdot (C\Delta^{s}|\mathcal{T}_{p}|), \end{displaymath}
which is equivalent to \eqref{form9a}.

The plan is now to apply Theorem \ref{thm2a} to the pair of measures $\mu_{\Delta}$ and $\nu_{\Delta}$, where $\nu_{\Delta}$ is a uniformly distributed measure on $\mathbb{T}$ giving unit mass to each $\mathbf{T} \in \mathbb{T}$. To make this precise, $\nu_{\Delta}$ actually needs to defined on the parameter set $[0,1] \times \R$ by
\begin{equation}\label{form61} \nu_{\Delta} := \sum_{\mathbf{q} \in \mathcal{Q}} \Delta^{-2} \mathbf{1}_{\mathbf{q}}, \end{equation}
where $\mathcal{Q} = \{\mathbf{q} \in \mathcal{D}_{\Delta}([0,1] \times \R) : T(\mathbf{q}) \in \mathbb{T}\}$. To find a lower bound for the weighted incidences between $\mu_{\Delta}$ and $\nu_{\Delta}$, the idea is simply that for each $\mathbf{p} \in \mathcal{P}_{\Delta} = \spt \mu_{\Delta}$, the family $\mathbb{T}_{\mathbf{p}} \subset \mathbb{T}$ has many elements thanks to \eqref{form9a}. To actually filter this idea through the formal definition of weighted incidences is slightly cumbersome. To do this carefully, let $\mathcal{Q}_{\mathbf{p}} := \{\mathbf{q} : T(\mathbf{q}) \in \mathbb{T}_{\mathbf{p}}\} \subset \mathcal{Q}$ be the parameter set of $\mathbb{T}_{\mathbf{p}}$, for $\mathbf{p} \in \mathcal{P}_{\Delta}$. The following is the key observation:
\begin{equation}\label{form57} \bigcup_{\mathbf{q} \in \mathcal{Q}_{\mathbf{p}}} \mathbf{q} \subset \{q : p \in T_{q}(20\Delta)\}, \qquad p \in \mathbf{p} \in \mathcal{P}_{\Delta}, \end{equation}
where $T_{q}(20\Delta)$ is the ordinary $20\Delta$-tube from Definition \ref{def:tube}. To see this, fix $p \in \mathbf{p} \in \mathcal{P}_{\Delta}$ and $\mathbf{q} \in \mathcal{Q}_{\mathbf{p}}$ and $q \in \mathbf{q}$. Then $T(\mathbf{q}) \in \mathbb{T}_{\mathbf{p}}$, so $T(\mathbf{q}) \cap \mathbf{p} \neq \emptyset$ by the definition of $\mathbb{T}_{\mathbf{p}}$. Since $\mathbf{p} \subset B(2)$, it follows from \eqref{form40} that $\dist(\mathbf{p},\ell_{q}) \leq 10\Delta$, and therefore $p \in \mathbf{p} \subset T_{q}(20\Delta)$. This is what was claimed.

With the inclusion above secured, the weighted incidences between $\mu_{\Delta}$ and $\nu_{\Delta}$ can be estimated easily:
\begin{align*} \mathcal{I}_{20\Delta}(\mu_{\Delta},\nu_{\Delta}) & \stackrel{\mathrm{def.}}{=} \int \nu_{\Delta}(\{q : p \in T_{q}(20\Delta)\}) \, d\mu_{\Delta}(p)\\
& = \sum_{\mathbf{p} \in \mathcal{P}_{\Delta}} \int_{\mathbf{p}} \nu_{\Delta}(\{q : p \in T_{q}(20\Delta)\}) \, d\mu_{\Delta}(p)\\
& \geq \sum_{\mathbf{p} \in \mathcal{P}_{\Delta}} \mu_{\Delta}(\mathbf{p})\nu\Big(\bigcup_{\mathbf{q} \in \mathcal{Q}_{\mathbf{p}}} \mathbf{q} \Big) = \sum_{\mathbf{p} \in \mathcal{P}_{\Delta}} \mu_{\Delta}(\mathbf{p})|\mathbb{T}_{\mathbf{p}}| \gtrsim \frac{\mu(\R^{2})\Delta^{-s}}{C\log^{4}(1/\Delta)}, \end{align*}
where the final inequality combined \eqref{form58} and \eqref{form9a}.

The plan is now to compare this lower bound against the upper bound provided by Theorem \ref{thm2a}. The $(3 - t)$-energy of $\nu_{\Delta}$ can be controlled by the Katz-Tao $(\Delta,\sigma + 1)$-set property of $\mathbb{T}$ (which is by definition the same as the Katz-Tao $(\Delta,\sigma + 1)$-set property of $\mathcal{Q}$), and recalling from the beginning of the proof that $\sigma + 1 > 3 - t$:
\begin{displaymath} I_{3 - t}(\nu_{\Delta}) \lesssim_{\sigma,t} |\mathbb{T}|\Delta^{-(\sigma + 1)} \stackrel{\textup{\nref{a}}}{\lesssim} (\epsilon \cdot \mu(\R^{2})/C^{3}) \cdot \Delta^{-2(\sigma + 1)}. \end{displaymath}
Recalling also \eqref{form11a}, and applying Theorem \ref{thm2a} at scale $20\Delta$ yields
\begin{align*} \frac{\mu(\R^{2})\Delta^{-s}}{C\log^{4}(1/\Delta)} & \lesssim_{\sigma,t} \Delta \sqrt{\epsilon \cdot C^{-2}\mu(\R^{2})^{2} \Delta^{-2(\sigma + 1)} \log(1/\Delta)}\\
& \leq \sqrt{\epsilon} \cdot \frac{\mu(\R^{2})}{C} \cdot \Delta^{-\sigma}\log(1/\Delta). \end{align*}
Rearranging this inequality leads to $\epsilon \gtrsim_{\sigma,t} \Delta^{2(\sigma - s)}\log^{-10}(1/\Delta)$, which contradicts the choice at \eqref{form12a} is $c = c(\sigma,t) > 0$ is small enough. This completes the proof. 

I omitted the minor detail that in order to apply Theorem \ref{thm2}, the measure $\nu_{\Delta}$ should actually be supported on $[\tfrac{1}{2},\tfrac{3}{4}] \times \R$. This can be arranged by initially reducing the proof of the current theorem to the case where angles of the tubes $\mathcal{T}_{p}$ lie in the interval $[\tfrac{1}{2},\tfrac{3}{4}]$.  \end{proof}

\section{A quantitative radial projection estimate}

The following result is \cite[Proposition 4.8]{2022arXiv220513890O}. This preprint was not published in original form, but the higher dimensional counterpart is \cite[Proposition 23]{2022arXiv220803597B}. 

 
 
  \begin{proposition}\label{prop2} For every $s \in (0,1]$, $t > 1$, and $\sigma \in (0,s)$, there exists $\delta_{0} = \delta_{0}(s,\sigma,t) > 0$ and $\epsilon = \epsilon(s,\sigma,t) > 0$ such that the following holds for all $\delta \in (0,\delta_{0}]$. Let $E,F \subset B(1) \subset \R^{2}$ be non-empty $\delta$-separated sets, where $E$ is a $(\delta,s,\delta^{-\epsilon})$-set, and $F$ is a $(\delta,t,\delta^{-\epsilon})$-set (recall Definition \ref{def:deltaSSet}), and $\dist(E,F) \geq \tfrac{1}{2}$. Then, there exists $q \in F$ such that
 \begin{equation}\label{form37} |\pi_{q}(E')|_{\delta} \geq \delta^{-\sigma}, \qquad E' \subset E, \, |E'| \geq \delta^{\epsilon}|E|. \end{equation}
 \end{proposition}
 
 
 
The proof of Theorem \ref{thm3} will require the following corollary:

 \begin{cor}\label{cor2} For every $s \in (0,1]$, $\sigma \in (0,s)$, $t \in (1,2]$, and $C > 0$, there exists a constant $K = K(C,s,\sigma,t) \geq 1$ such that the following holds.
 
 Let $\mu,\nu$ be Borel probability measures supported on $B(1)$, with $\dist(\spt \mu,\spt \nu) \geq \tfrac{1}{2}$. Assume 
 \begin{equation}\label{form35} \mu(B(x,r)) \leq Cr^{t} \quad \text{and} \quad \nu(B(x,r)) \leq Cr^{s} \end{equation}
 for all $x \in \R^{2}$ and $r > 0$. Then, there exists $q \in \spt \mu$ such that 
 \begin{displaymath} \mathcal{H}^{\sigma}_{\infty}(\pi_{q}(E')) \geq \tfrac{1}{K}(\nu(E'))^{K} \end{displaymath}
 for all Borel sets $E' \subset \spt \nu$.   \end{cor}
 
 \begin{remark}\label{rem1} In fact, the proof shows that "$q \in \spt \mu$" can be replaced by "$q \in F$" where $F \subset \spt \mu$ is any Borel set with $\mu(F) = 1$.  \end{remark}
 
 \begin{proof}[Proof of Corollary \ref{cor2}] The plan is to apply Proposition \ref{prop2} with parameters $s,t$ and $\sigma$, so let $\epsilon > 0$ and $\delta_{0} > 0$ be the constants provided by the proposition with these parameters. The constant $K \geq 1$ will need to be so large that the following requirements are met:
\begin{itemize}
\item $K^{-1} \leq \min\{\delta_{0},\epsilon/10\}$,
\item $AC\log^{A}(1/\Delta) \leq \Delta^{-\epsilon/4}$ for all $\Delta \leq K^{-1}$, for a suitable absolute constant $A \geq 1$,
\end{itemize} 
 
 The proof starts with a counter assumption: for every $q \in F := \spt \mu$ there exists a Borel set $E_{q} \subset E := \spt \nu$ such that
 \begin{displaymath} \mathcal{H}^{\sigma}_{\infty}(\pi_{q}(E_{q})) < \tfrac{1}{K}(\nu(E_{q}))^{K}, \end{displaymath}
in particular $\nu(E_{q}) > 0$. Write $2^{-j} =: \Delta_{j}$ for $j \geq 1$. Using the definition of Hausdorff content, find families of $\Delta_{j}$-tubes $\mathcal{T}^{j}_{q}$ containing $q$ such that
\begin{equation}\label{form32} E_{q} \subset \bigcup_{j \geq 1} \cup \mathcal{T}^{j}_{q} \quad \text{and} \quad \sum_{j \geq 1} \Delta_{j}^{\sigma}|\mathcal{T}^{j}_{q}| \leq \tfrac{1}{K}(\nu(E_{q}))^{K} \leq \tfrac{1}{K}. \end{equation}
(In this proof it is not needed that the families $\mathcal{T}_{q}^{j}$ are Katz-Tao $(\Delta_{j},\sigma)$-sets.) By the pigeonhole principle, one may for each $q \in F$ pick an index $j = j(q) \geq 1$ such that
\begin{equation}\label{form33} \nu(E_{q} \cap (\cup \mathcal{T}_{q}^{j})) \gtrsim \frac{\nu(E_{q})}{j^{2}} > 0. \end{equation}
Note that, by the final inequality in \eqref{form32}, and since $\sigma \leq 1$, one may deduce that $\mathcal{T}_{q}^{j} = \emptyset$ for all $j \in \{1,\ldots, \log K\}$. Therefore $j(q) \geq \log K$, or in other words $\Delta_{j(q)} \leq K^{-1} \leq \delta_{0}$.

By a second application of the pigeonhole principle, find an index $j \geq 1$ such that the set $F_{j} := \{q \in F : j(q) = j\}$ satisfies $\mu(F_{j}) \gtrsim 1/j^{2}$. Fix this index $j \geq 1$ for the rest of the proof, and abbreviate 
\begin{displaymath} \Delta := \Delta_{j} \leq \delta_{0}, \quad \mathcal{T}_{q} := \mathcal{T}_{q}^{j} \quad \text{and} \quad \bar{F} := F_{j}. \end{displaymath}

From \eqref{form33}, one sees that $\mathcal{T}_{q} \neq \emptyset$ for all $q \in \bar{F}$, therefore $|\mathcal{T}_{q}| \geq 1$. One may combine this information with \eqref{form32} to find
\begin{displaymath} \Delta^{\sigma} \leq \Delta^{\sigma} |\mathcal{T}_{q}| \leq (\nu(E_{q}))^{K} \quad \Longrightarrow \quad \nu(E_{q}) \geq \Delta^{\sigma/K} \geq \Delta^{1/K} \geq \Delta^{\epsilon/10}, \end{displaymath}
and consequently, by the choice of $K$,
\begin{equation}\label{form60} \nu(E_{q} \cap (\cup \mathcal{T}_{q})) \stackrel{\eqref{form33}}{\gtrsim} \frac{\nu(E_{q})}{\log^{2}(1/\Delta)} \geq \frac{\Delta^{1/K}}{\log^{2}(1/\Delta)} \geq \Delta^{\epsilon/8}.\end{equation}
Consider the family $\mathcal{F}_{\Delta} := \mathcal{D}_{\Delta}(\bar{F})$. Evidently $\mu(\cup \mathcal{F}_{\Delta}) \geq \mu(\bar{F}) \gtrsim \log^{-2}(1/\Delta)$. For each $\mathbf{q} \in \mathcal{F}_{\Delta}$, fix one distinguished point $q \in \bar{F} \cap \mathbf{q}$, and let
\begin{displaymath} \mathcal{E}_{\mathbf{q}} := \cup \{\mathbf{p} \in \mathcal{D}_{\Delta}(E) : \mathbf{p} \cap (\cup \mathcal{T}_{q}) \neq \emptyset\} \supset E_{q} \cap (\cup \mathcal{T}_{q}). \end{displaymath} 
Recalling that $\mathcal{T}_{q}$ is a collection of $\Delta$-tubes containing $q$, hence intersecting $\mathbf{q}$, one has
\begin{equation}\label{form34} |\pi_{\mathbf{q}}(\mathcal{E}_{\mathbf{q}})|_{\Delta} \lesssim |\mathcal{T}_{q}| \stackrel{\eqref{form32}}{\leq} \tfrac{1}{K}\Delta^{-\sigma} < \Delta^{-\sigma}. \end{equation}
(This also uses $\dist(E,F) \geq \tfrac{1}{2}$.) Furthermore, $\nu(\mathcal{E}_{\mathbf{q}}) \gtrsim \Delta^{\epsilon/8}$ for all $\mathbf{q} \in \mathcal{F}_{\Delta}$, by \eqref{form60}.

To summarise, we have now located a set $\mathcal{F}_{\Delta} \subset \mathcal{D}_{\Delta}$ with $\mu(\cup \mathcal{F}_{\Delta}) \gtrsim \log^{-2}(1/\Delta)$ such that for every $\mathbf{q} \in \mathcal{F}_{\Delta}$, there exists a set $\mathcal{E}_{\mathbf{q}} \subset \mathcal{D}_{\Delta}(E)$ with $\nu(\cup \mathcal{E}_{\mathbf{q}}) \gtrsim \Delta^{\epsilon/8}$ and small $\Delta$-scale radial projection to $\mathbf{q}$ in the sense \eqref{form34}. This nearly contradicts Proposition \ref{prop2} applied at scale $\Delta$, except that the measures $\nu$ and $\mu$ need to replaced by a $(\Delta,s,\Delta^{-\epsilon})$-set and a $(\Delta,t,\Delta^{-\epsilon})$-set, respectively. This is routine pigeonholing, but I sketch the details.

Start by pigeonholing a subset $\mathcal{F}_{\Delta}' \subset \mathcal{F}_{\Delta}$ such that 
\begin{itemize}
\item $\mu(\mathbf{q})/\mu(\mathbf{q}') \in [\tfrac{1}{2},2]$ for $\mathbf{q},\mathbf{q}' \in \mathcal{F}_{\Delta}'$, and
\item $\mu(\cup \mathcal{F}_{\Delta}') \gtrsim \log^{-3}(1/\Delta)$.
\end{itemize}
It now follows from the Frostman condition \eqref{form35} on $\mu$ that $\mathcal{F}_{\Delta}'$ is a $(\Delta,t,O(C\log^{3}(1/\Delta)))$-set. In particular $\mathcal{F}_{\Delta}'$ is a $(\Delta,t,\Delta^{-\epsilon/2})$-set by the choice of $K$.

Next, for each $\mathbf{q} \in \mathcal{F}_{\Delta}'$, find by a similar procedure (now using the $s$-Frostman property of $\nu$) a $(\Delta,s,\Delta^{-\epsilon/4})$-subset $\bar{\mathcal{E}}_{\mathbf{q}} \subset \mathcal{E}_{\mathbf{q}}$ with $\nu(\cup \bar{\mathcal{E}}_{\mathbf{q}}) \geq \Delta^{\epsilon/4}$, and such that $\mathbf{p} \mapsto \nu(\mathbf{p})$ is almost constant (depending on $\mathbf{q}$) for $\mathbf{p} \in \bar{\mathcal{E}}_{\mathbf{q}}$. Finally, verify (using Cauchy-Schwarz) that some fixed subset $\mathcal{E} := \bar{\mathcal{E}}_{\mathbf{q}_{0}}$ satisfies 
\begin{displaymath} \nu(\mathcal{E} \cap \bar{\mathcal{E}}_{\mathbf{q}}) \geq \Delta^{\epsilon}, \qquad \mathbf{q} \in \mathcal{F}_{\Delta}'', \end{displaymath}
where $\mathcal{F}_{\Delta}'' \subset \mathcal{F}_{\Delta}'$ has $|\mathcal{F}_{\Delta}''| \geq \Delta^{\epsilon/2}|\mathcal{F}_{\Delta}'|$. Now $\mathcal{F}_{\Delta}''$ is a $(\Delta,t,\Delta^{-\epsilon})$-set, $\mathcal{E}$ is a $(\Delta,s,\Delta^{-\epsilon})$-set, and for each $\mathbf{q} \in \mathcal{F}_{\Delta}''$ the fat subset $\mathcal{E} \cap \bar{\mathcal{E}}_{\mathbf{q}} \subset \mathcal{E}$ continues to satisfy \eqref{form34}. This contradicts Proposition \ref{prop2} at scale $\Delta \leq \delta_{0}$, and completes the proof of the corollary.   \end{proof}

\section{Proof of the discretised main theorem}

In this section the $\delta$-discretised Theorem \ref{thm3} is proven. Here is the statement again:

\begin{thm}\label{thm4} For every $s \in (0,1]$ and $t \in (1,2]$ such that $s + t > 2$, for every $\tau \in (1,t)$, and $C \geq 1$, there exists a constants $\delta_{0} = \delta_{0}(C,s,t,\tau) > 0$ and $\epsilon = \epsilon(C,s,t,\tau) > 0$ such that the following holds for all $\delta \in 2^{-\N} \cap (0,\delta_{0}]$.

Let $\mu,\nu$ be Borel probability measures supported on $B(1)$, with $\dist(\spt \mu,\spt \nu) \geq \tfrac{1}{2}$. Assume
\begin{displaymath} \mu(B(x,r)) \leq Cr^{t} \quad \text{and} \quad \nu(B(x,r)) \leq Cr^{s}, \qquad x \in \R^{2}, \, r > 0. \end{displaymath}
Let $E \subset \spt \nu$ and $F \subset \spt \mu$ be Borel sets with full $\nu$ measure and $\mu$ measure, respectively.

Assume that to every $p \in E$ there corresponds a family $\mathcal{T}_{p} \subset \mathcal{T}^{\delta}$ of dyadic $\delta$-tubes such that $B(p,2\delta) \cap T \neq \emptyset$ for all $T \in \mathcal{T}_{p}$, and
\begin{equation}\label{form13} \mu(\cup \mathcal{T}_{p}) \geq C^{-1}, \qquad p \in E. \end{equation}
Then, there exists $p \in E$ and a tube $T \in \mathcal{T}_{p}$ such that $\mathcal{H}^{\tau - 1}_{\delta,\infty}(F \cap T) \geq \epsilon$.  \end{thm}

\begin{remark} There are the following minor differences compared to the first version of Theorem \ref{thm3}: the families $\mathcal{T}_{p}$ now consist of dyadic $\delta$-tubes instead of ordinary $\delta$-tubes, the requirement $p \in T$ is relaxed to $T \cap B(p,2\delta) \neq \emptyset$, and the conclusion is about $F \cap T$ instead of $F \cap 10T$. It is routine to pass between these two forms of the theorem. (Hint: check that if $T$ is an ordinary $\delta$-tube containing $p$, then $T$ can be covered by $\lesssim 1$ dyadic $\delta$-tubes $T_{1},\ldots,T_{N}$ intersecting $B(p,2\delta)$ and such that $F \cap T_{j} \subset F \cap 10T$ for $1 \leq j \leq N$.) The formulation of Theorem \ref{thm4} is more convenient to prove: for example, automatically $|\mathcal{T}_{p}| \lesssim \delta^{-1}$ for all $p \in E$, and summing over $T \in \mathcal{T}_{p}$ makes sense. \end{remark}

\begin{proof}[Proof of Theorem \ref{thm4}] Fix $\underline{t} \in (1,t)$ and $\underline{s} < s$ so close to $t,s$, respectively, that 
\begin{equation}\label{form29} \underline{s} > 2 - \underline{t}. \end{equation} 
We treat $\underline{s},\underline{t}$ as functions of $s,t$: dependence on $\underline{s},\underline{t}$ is recorded as a dependence on $s,t$.
 
The proof starts by restricting $\mu$ to a "good set" $F' \subset \spt \mu$ of measure $1 - (2C)^{-1}$. Namely, let $K = K(2C^{2},s,\underline{s},t) \geq 1$ be the constant from Corollary \ref{cor2} with the parameters listed in parentheses. I claim that there exists a subset $F' \subset \spt \mu$ with $\mu(F') \geq 1 - (2C)^{-1}$ and the following property. If $q \in F'$, then
\begin{equation}\label{form2} \mathcal{H}^{\underline{s}}_{\infty}(\pi_{q}(E')) \geq \tfrac{1}{K}(\nu(E'))^{K}, \qquad E' \subset \spt \nu \text{ Borel}. \end{equation}
Let $F_{\mathrm{bad}} \subset \spt \mu$ be the set of points $q \in \spt \mu$ failing \eqref{form2}. If $\mu(F_{\mathrm{bad}}) > (2C)^{-1}$, then the renormalised measure $\bar{\mu} := \mu(F_{\mathrm{bad}})^{-1}\mu|_{F_{\mathrm{bad}}}$ is a probability measure satisfying the Frostman condition $\bar{\mu}(B(x,r)) \leq 2C^{2}r^{t}$. One may then apply Corollary \ref{cor2} to $\bar{\mu},\nu$. By the choice of the constant $K = K(2C^{2},s,\underline{s},t)$, one finds a point $q \in F_{\mathrm{bad}}$ (a set of full $\bar{\mu}$ measure) satisfying \eqref{form2}. This contradicts the definition of $F_{\mathrm{bad}}$.

Define $F' := \spt \mu \, \setminus \, F_{\mathrm{bad}}$. It was established above that $\mu(F') \geq 1 - (2C)^{-1}$. In particular, by \eqref{form13}, 
\begin{displaymath} \mu(F' \cap (\cup \mathcal{T}_{p})) \geq (2C)^{-1}, \qquad p \in \mathcal{E}. \end{displaymath}
This observation shows that the measures $\mu$ and $\mu|_{F'}$ satisfy the same hypotheses, up to replacing "$C$" by "$2C$". Since it also suffices to show that $\mathcal{H}^{\tau - 1}_{\delta,\infty}(F' \cap T) \geq \epsilon$ for some $p \in E$ and $T \in \mathcal{T}_{p}$, it is now safe to restrict $\mu$ to $F'$. I do this without altering notation: in other words I assume in the sequel that $\mu(F') = 1$ and \eqref{form2} holds for all $q \in F$. 

Now comes main counter assumption: 
\begin{equation}\label{form18} \mathcal{H}^{\tau - 1}_{\delta,\infty}(F \cap T) \leq \epsilon, \qquad p \in E, \, T \in \mathcal{T}_{p}. \end{equation}
Here
\begin{equation}\label{form31} \epsilon := \tfrac{1}{\mathbf{A}} \inf \{ \log^{-\mathbf{A}}(1/\Delta)\Delta^{\tau - t} : \Delta \in (0,\tfrac{1}{2}]\} \end{equation}
and $\mathbf{A} = \mathbf{A}(C,s,t,\tau) \geq 1$ is a constant to be determined at the very end of the proof. Note that $\epsilon \gtrsim_{C,s,t,\tau} 1$ since $\tau < t$. For the remainder of the argument, I use the convention that the constants in the "$\lesssim$" notation are allowed to depend on $C,s,t,\tau$.

The counter assumption \eqref{form18} and the pigeonhole principle will next be applied to find useful, relatively large, subsets of $E$ and $F$. Start by fixing $p \in E$ and $T \in \mathcal{T}_{p}$. According to \eqref{form18}, one can find a family of rectangles $\mathcal{R}(p,T)$ which are contained in $T$, cover $F \cap T$, and satisfy 
\begin{equation}\label{form19} \sum_{R \in \mathcal{R}(p,T)} \diam(R)^{\tau - 1} \leq \epsilon. \end{equation}
Perhaps enlarging "$\epsilon$" slightly, one may assume the rectangles in $\mathcal{R}(p,T)$ have dimensions $\delta \times 2^{-j}$ for some $\delta \leq 2^{-j} \leq \tfrac{1}{2}$. Let us group the rectangles according to their longer side: 
\begin{displaymath} \mathcal{R}(p,T,j) := \{R \in \mathcal{R}(p,T) : R \text{ is a } \delta \times 2^{-j}\text{-rectangle}\}, \end{displaymath}
for $j \in \{1,\ldots,\log(1/\delta)\}$. With this notation, observe that
\begin{displaymath} \sum_{j = 1}^{\log(1/\delta)} \Big( \int_{E} \sum_{T \in \mathcal{T}_{p}} \sum_{R \in \mathcal{R}(p,T,j)} \mu(R) \, d\nu p \Big) \geq \int_{E} \mu(\cup \mathcal{T}_{p}) \, d\nu p \geq \frac{1}{C}, \end{displaymath}
according to \eqref{form13}. It follows that there exists an index $j \in \{1,\ldots,\log(1/\delta)\}$ such that
\begin{equation}\label{form23} \int_{E} \sum_{T \in \mathcal{T}_{p}} \sum_{R \in \mathcal{R}(p,T,j)} \mu(R) \, d\nu p \geq \frac{c}{Cj^{2}}. \end{equation}
This index "$j$" will be fixed for the remainder of the proof. Abusing notation, I abbreviate $\mathcal{R}(p,T) := \mathcal{R}(p,T,j)$ for this index "$j$", and write 
\begin{displaymath} \Delta := 2^{-j} \in [\delta,\tfrac{1}{2}]. \end{displaymath}
Estimate \eqref{form23} further implies that there exists a subset $\bar{E} \subset E$ with $\nu(\bar{E}) \geq c/(C\log^{2}(1/\Delta))$ and the property
\begin{equation}\label{form24} \sum_{T \in \mathcal{T}_{p}} \sum_{R \in \mathcal{R}(p,T)} \mu(R) \geq \frac{c}{C\log^{2} (1/\Delta)}, \qquad p \in \bar{E}. \end{equation}
The families $\mathcal{T}_{p}$ and $\mathcal{R}(p,T)$ still need pruning. Fix $p \in \bar{E}$, and note that $|\mathcal{T}_{p}| \lesssim \delta^{-1}$, since $\mathcal{T}_{p}$ is a family of dyadic $\delta$-tubes intersecting $B(p,2\delta)$. A tube $T \in \mathcal{T}_{p}$ is called \emph{heavy}, denoted $T \in \mathcal{T}_{p,\mathrm{heavy}}$, if
\begin{equation}\label{form25} \sum_{R \in \mathcal{R}(p,T)} \mu(R) \geq \frac{c'\delta}{C\log^{2}(1/\Delta)} \end{equation}
for a suitable constant $c' > 0$. Since $|\mathcal{T}_{p}| \lesssim \delta^{-1}$, it follows from \eqref{form24} that
\begin{displaymath} \sum_{T \in \mathcal{T}_{p,\mathrm{heavy}}} \sum_{R \in \mathcal{R}(p,T)} \mu(R) \geq \tfrac{1}{2} \sum_{T \in \mathcal{T}_{p}} \sum_{R \in \mathcal{R}(p,T)} \mu(R), \qquad p \in \bar{E}. \end{displaymath} 
Since this lower bound is just as good as \eqref{form24} for future purposes, I will assume that $\mathcal{T}_{p} = \mathcal{T}_{p,\mathrm{heavy}}$ for all $p \in \bar{E}$, or in other words every tube $T \in \mathcal{T}_{p}$ satisfies \eqref{form25}. 

Finally, for $p \in E$ and $T \in \mathcal{T}_{p} = \mathcal{T}_{p,\mathrm{heavy}}$ fixed, the family $\mathcal{R}(p,T)$ needs pruning. Observe from \eqref{form19} that $|\mathcal{R}(p,T)| \leq \epsilon \cdot \Delta^{1 - \tau}$. A rectangle $R \in \mathcal{R}(p,T)$ is called heavy, denoted $R \in \mathcal{R}(p,T)_{\mathrm{heavy}}$, if 
\begin{equation}\label{form26} \mu(R) \geq \frac{c''\delta}{C\epsilon \cdot \Delta^{1 - \tau}\log^{2}(1/\Delta)} \end{equation}
for a suitable constant $c'' > 0$. Since $|\mathcal{R}(p,T)| \leq \epsilon \cdot \Delta^{1 - \tau}$, it follows from \eqref{form25} that
\begin{displaymath} \sum_{R \in \mathcal{R}(p,T)_{\mathrm{heavy}}} \mu(R) \geq \tfrac{1}{2} \sum_{R \in \mathcal{R}(p,T)} \mu(R), \qquad T \in \mathcal{T}_{p}. \end{displaymath} 
Again, this lower bound for the sum over $R \in \mathcal{R}(p,T)_{\mathrm{heavy}}$ is just as useful as the lower bound for the full sum over $R \in \mathcal{R}(p,T)$, so I assume that $\mathcal{R}(p,T)_{\mathrm{heavy}} = \mathcal{R}(p,T)$ to begin with, for all $p \in \bar{E}$ and $T \in \mathcal{T}_{p}$.

Here is a summarise of the achievements so far. We have found:
\begin{itemize}
\item[(i) \phantomsection \label{i}] A set $\bar{E} \subset E$ satisfying $\nu(\bar{E}) \gtrsim 1/(C\log^{2}(1/\Delta))$.
\item[(ii) \phantomsection \label{ii}] For each $p \in \bar{E}$ the tubes $T \in \mathcal{T}_{p}$ cover substantial $\mu$ measure, according to \eqref{form24}. In fact, 
\begin{displaymath} \mu(F_{p}) \gtrsim \frac{1}{C\log^{2}(1/\Delta)}, \quad \text{where} \quad F_{p} := \bigcup_{T \in \mathcal{T}_{p}} \bigcup_{R \in \mathcal{R}(p,T)} R. \end{displaymath}
\item[(iii) \phantomsection \label{iii}] For $p \in \bar{E}$ and $T \in \mathcal{T}_{p}$, every rectangle $R \in \mathcal{R}(p,T)$ is heavy, thus satisfying \eqref{form26}.
\end{itemize}

From now on the following notational convention will be in place: for the scale $\Delta \in [\delta,\tfrac{1}{2}]$ located above, I will write $A \lessapprox_{\Delta} B$ if 
\begin{displaymath} A \leq \mathbf{C}\log^{\mathbf{C}}(1/\Delta)B \end{displaymath}
for some constant $\mathbf{C} \geq 1$ which may depend on $C,s,t,\tau$. Any $\Delta$-dependence of this form is harmless. (In contrast, constants of the form $\log(1/\delta)$ would be detrimental.)

Observe that
\begin{displaymath} \int_{F} \nu(\{p \in \bar{E} : q \in F_{p}\}) \, d\mu q  = \int_{\bar{E}} \mu(F_{p}) \, d\nu p \gtrapprox_{\Delta} 1.  \end{displaymath}
by properties \nref{i} and \nref{ii} above. Consequently, there exists a set $G \subset F$ such that $\mu(G) \gtrapprox_{\Delta} 1$, and 
\begin{equation}\label{form27} \nu(\{p \in \bar{E} : q \in F_{p}\}) \gtrapprox_{\Delta} 1, \qquad q \in G. \end{equation} 
Write $E_{q} := \{p \in \bar{E} : q \in F_{p}\}$ for $q \in G$.

I claim that it is possible to select a distinguished square $Q \in \mathcal{D}_{\Delta}$ with the property
\begin{equation}\label{form4} \mu(G \cap Q) \gtrapprox_{\Delta} \mu(10Q). \end{equation}
Indeed, if this fails for all squares $Q \in \mathcal{D}_{\Delta}$, then the bounded overlap of the squares $10Q$ would immediately contradict $\mu(G) \gtrapprox_{\Delta} 1$. Now, fix a square $Q \in \mathcal{D}_{\Delta}$ satisfying \eqref{form4} for the remainder of the proof, see Figure \ref{fig1}.

\begin{figure}[h!]
\begin{center}
\begin{overpic}[scale = 1.2]{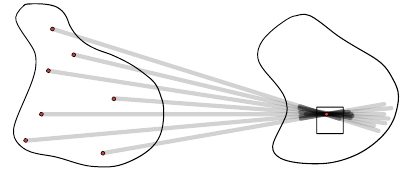}
\put(5,37){$E$}
\put(20,18){$E_{q}$}
\put(80,5){$Q$}
\put(80,11){\small{$q$}}
\put(70,33){$F$}
\end{overpic}
\caption{The square $Q$, a point $q \in G \cap Q$, the set $E_{q}$, and the tube family $\mathcal{T}_{0}^{q}$. The darker patches around $q$ signify the $(\delta \times \Delta)$-rectangles $R(T) \in \mathcal{R}(p,T)$, $p \in E_{q}$, all of which contain $q$ and are contained in $10Q$. }\label{fig1}
\end{center}
\end{figure}
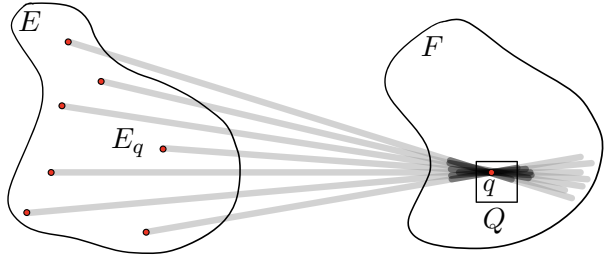

Fix $q \in G \cap Q \subset F$, and recall that $\nu(E_{q}) \gtrapprox_{\Delta} 1$ by \eqref{form27}. Consequently, according to \eqref{form2} applied to $E' = E_{q}$, and recalling that $K = K(2C^{2},s,\underline{s},t)$ is a constant depending on $C,s,t$ only (since $\underline{s}$ is a function of $s,t$),
\begin{equation}\label{form36} \mathcal{H}^{\underline{s}}_{\infty}(\pi_{q}(E_{q})) \geq \tfrac{1}{K}(\nu(E_{q}))^{K} \gtrapprox_{\Delta} 1. \end{equation}
To draw benefit from \eqref{form36}, recall that $E_{q} = \{p \in \bar{E} : q \in F_{p}\}$, where further
\begin{displaymath} F_{p} = \bigcup_{T \in \mathcal{T}_{p}} \bigcup_{R \in \mathcal{R}(p,T)} R. \end{displaymath}
Thus, for every $p \in E_{q}$, there exists a tube $T = T(p,q) \in \mathcal{T}_{p}$ and a rectangle $R = R(T) \in \mathcal{R}(p,T)$ with 
\begin{displaymath} q \in R \subset T \quad \text{and} \quad T \cap B(p,2\delta) \neq \emptyset. \end{displaymath}
The family $\mathcal{T}^{q}_{0} := \{T(p,q) : p \in E_{q}\}$ now consists of $\delta$-tubes containing $q$, and $E_{q}$ is contained in the $(2\delta)$-neighbourhood of the union $\cup \mathcal{T}^{q}_{0}$, see Figure \ref{fig1}. The lower bound \eqref{form36} says (slightly informally) that the $\mathcal{H}^{\underline{s}}_{\delta,\infty}$-content of $\mathcal{T}^{q}_{0}$ is bounded from below by $\gtrapprox_{\Delta} 1$. Therefore, applying \cite[Proposition A.1]{FaO} at scale $\delta/\Delta \geq \delta$, one can select a $(\delta/\Delta,\underline{s},\approx_{\Delta} 1)$-set 
\begin{displaymath} \mathcal{T}^{q} \subset \mathcal{T}^{q}_{0} \end{displaymath}
consisting of $(\delta/\Delta)$-separated elements (since the tubes in $\mathcal{T}_{0}^{q}$ contain $q$, this is equivalent to saying that the angles of the tubes in $\mathcal{T}^{q}$ are $(\delta/\Delta)$-separated).

For each $T = T(p,q) \in \mathcal{T}^{q}$, recall that $q \in R(T) \subset T$. Since $q \in Q$, and $\diam(R(T)) \leq 2\Delta$, and $Q$ is a square of side $\Delta$, it holds that $R(T) \subset 10Q$. Consequently, one may infer from \eqref{form26}, and \nref{iii}, that
\begin{equation}\label{form5} \mu(T \cap 10Q) \geq \mu(R(T)) \gtrapprox_{\Delta} \epsilon^{-1}(\delta/\Delta)\Delta^{-\tau}, \qquad T \in \mathcal{T}^{q}. \end{equation}
Informally speaking, every tube in $\mathcal{T} := \bigcup_{q \in G \cap Q} \mathcal{T}^{q}$ has many incidences with $\mu|_{10Q}$.

Before attempting to formalise this statement properly, let us examine the tube family $\mathcal{T}$. For $q \in G \cap Q$, the family $\mathcal{T}^{q}$ was by definition a $(\delta/\Delta)$-separated $(\delta/\Delta,\underline{s},\approx_{\Delta} 1)$-set of $\delta$-tubes incident to $q$. The morale for "$\delta/\Delta$" is that we are seeking upper and lower bounds on incidences between $\mathcal{T}$ and $\mu|_{10Q}$, and it is natural to "renormalise" the problem by rescaling $\mu|_{10Q}$ by a factor of $\Delta^{-1}$.

In fact, consider the measure $\mu_{Q}$ supported on $[0,1]^{2}$, defined by
\begin{displaymath} \boldsymbol{\mu} := \Delta^{-t}S_{10Q}(\mu|_{10Q}), \end{displaymath}
where $S_{10Q}$ is the scaling map taking $10Q$ to $[0,1]^{2}$. Note that the normalisation by $\Delta^{-t}$ preserves the $t$-Frostman condition, that is, $\mu_{Q}(B(x,r)) \lesssim Cr^{t}$ for all $x \in \R^{2}$ and $r > 0$.

For $\mathbf{q} = S_{10Q}(q) \in S_{10Q}(G \cap Q)$, consider also the rescaled tube family 
\begin{displaymath} \mathcal{T}^{\mathbf{q}} := \{S_{10Q}(T) : T \in \mathcal{T}^{q}\}. \end{displaymath}
Rescaling does not affect angular separation, but simply thickens the tubes: $\mathcal{T}^{\mathbf{q}}$ is a $(\delta/\Delta)$-separated $(\delta/\Delta,\underline{s},\approx_{\Delta} 1)$-set of $(\delta/\Delta)$-tubes incident to $\mathbf{q}$, for $\mathbf{q} \in S_{10Q}(G \cap Q) =: \boldsymbol{G}$. 

To be accurate, whereas $\mathcal{T}^{q}$ consists of dyadic $\delta$-tubes, the rescaled sets $S_{10Q}(T)$ are not exactly dyadic $(\delta/\Delta)$-tubes. To mend this, I associate to each set $S_{10Q}(T) \in \mathcal{T}^{\mathbf{q}}$ a suitable dyadic $(\delta/\Delta)$-tube. The important properties of the sets $S_{10Q}(T) \in \mathcal{T}^{\mathbf{q}}$ are: 
\begin{itemize}
\item[(a)] they contain $\mathbf{q} = S_{10Q}(q)$, and
\item[(b)] $\boldsymbol{\mu}(S_{10Q}(T)) = \Delta^{-t}\mu(T \cap 10Q) \gtrapprox_{\Delta} \epsilon^{-1}(\delta/\Delta)\Delta^{\tau - t}$ by \eqref{form5}. 
\end{itemize}
It is desirable to preserve these properties when replacing $S_{10Q}(T)$ by a dyadic $(\delta/\Delta)$-tube. Fixing $T \in \mathcal{T}^{q}$, cover $S_{10Q}(T) \cap [0,1]^{2}$ by $\lesssim 1$ dyadic $(\delta/\Delta)$-tubes. Then, at least one of them, say $\mathbf{T}$, still satisfies 
\begin{equation}\label{form38} \boldsymbol{\mu}(\mathbf{T}) \gtrapprox_{\Delta} \epsilon^{-1}(\delta/\Delta)\Delta^{\tau - t}, \end{equation}
and also $\mathbf{T} \cap B(\mathbf{q},A(\delta/\Delta)) \neq \emptyset$ for a suitable absolute constant $A \geq 1$; this is a little weaker than (a), but perfectly adequate. Redefine $\mathcal{T}^{\mathbf{q}}$ to consist of the dyadic $(\delta/\Delta)$-tubes obtained in this way.

Let us record that
\begin{equation}\label{form21} \boldsymbol{\mu}(\boldsymbol{G}) = \Delta^{-t}\mu(G \cap Q) \gtrapprox_{\Delta} \Delta^{-t}\mu(10Q) \end{equation}
by \eqref{form4}. On the other hand, using that $\mu_{Q}(B(x,r)) \lesssim Cr^{t}$ and $\underline{t} < t$,  
\begin{equation}\label{form28} I_{\underline{t}}(\boldsymbol{\mu}) \lesssim \boldsymbol{\mu}(\R^{2}) = \Delta^{-t}\mu(10Q). \end{equation}
It is crucial that the same factor $\mu(10Q)$ appears in both bounds \eqref{form21}-\eqref{form28}.

Recall from \eqref{form29} that $\underline{s} > 2 - \underline{t} > 2 - t$. Fix $\sigma < \underline{s}$ such that still 
\begin{equation}\label{form41} \sigma > 2 - \underline{t}. \end{equation}
Recall again that for each $\mathbf{q} \in \boldsymbol{G}$, the family $\mathcal{T}^{\mathbf{q}} \subset \mathcal{T}^{\delta/\Delta}$ is a $(\delta/\Delta,\underline{s},\approx_{\Delta} 1)$-set. Apply Theorem \ref{thm3a} to the families $\mathcal{T}^{\mathbf{q}}$, $\mathbf{q} \in \boldsymbol{G}$, to the measure $\boldsymbol{\mu}|_{\boldsymbol{G}}$, and with parameters $t,\underline{s},\sigma$. The conclusion is that the union $\bigcup_{\mathbf{q} \in \boldsymbol{G}} \mathcal{T}^{\mathbf{q}}$ contains a Katz-Tao $(\delta/\Delta,\sigma + 1)$-set $\mathcal{T}$ of cardinality 
\begin{equation}\label{form20} |\mathcal{T}| \gtrapprox_{\Delta} \boldsymbol{\mu}(\boldsymbol{G}) \left(\frac{\delta}{\Delta} \right)^{-(\sigma + 1)} \stackrel{\eqref{form21}}{\gtrsim} \left(\frac{\delta}{\Delta} \right)^{-(\sigma + 1)}\Delta^{-t}\mu(10Q). \end{equation}

Let $\nu$ be $(\delta/\Delta)$-discretised measure on $\mathcal{T} \subset \mathcal{T}^{\delta/\Delta}$ which gives unit mass to each element $T \in \mathcal{T}$. Formally, much like in \eqref{form61}, $\nu$ is defined as an absolutely continuous measure on the "parameter" space $[0,1] \times \R$ of dyadic $(\delta/\Delta)$-tubes:
\begin{displaymath} \boldsymbol{\nu} = \sum_{Q \in \mathcal{Q}} (\tfrac{\delta}{\Delta})^{-2}\mathbf{1}_{Q}, \end{displaymath}
where $\mathcal{Q} = \{Q \in \mathcal{D}_{\delta/\Delta}([0,1] \times \R) : T(Q) \in \mathcal{T}\}$ is the collection of "parameter squares" for the tubes in $\mathcal{T}$. Recall from \eqref{form40} that
\begin{equation}\label{form39} T(Q) \cap B(2) \subset \ell_{\theta,r}(10(\delta/\Delta)), \qquad (\theta,r) \in Q. \end{equation}
Since $\sigma + 1 > 3 - \underline{t}$ by \eqref{form41}, and $\mathcal{Q}$ is a Katz-Tao $(\delta/\Delta,\sigma + 1)$-set, it is straightforward to check that $\int d\boldsymbol{\nu}(y)/|x - y|^{(3 - \underline{t})} \lesssim (\delta/\Delta)^{-(\sigma + 1)}$ for all $x \in \R^{2}$. Therefore
\begin{equation}\label{form22} I_{3 - \underline{t}}(\boldsymbol{\nu}) \lesssim \left(\tfrac{\delta}{\Delta} \right)^{-(\sigma + 1)}\boldsymbol{\nu}(\R^{2}) = \left(\tfrac{\delta}{\Delta} \right)^{-(\sigma + 1)}|\mathcal{T}|. \end{equation}
Next, recall from \eqref{form38} that the tubes in $\mathcal{T}$ have many $(\delta/\Delta)$-incidences with $\boldsymbol{\mu}$. In fact,
\begin{align}\label{form30} \mathcal{I}_{10(\delta/\Delta)}(\boldsymbol{\mu},\boldsymbol{\nu}) & \stackrel{\eqref{def:incidences}}{=} \int \boldsymbol{\mu}(\ell_{\theta,r}(10(\delta/\Delta)) \, d\boldsymbol{\nu}(\theta,r) \notag\\
& = \sum_{Q \in \mathcal{Q}} \left(\tfrac{\delta}{\Delta} \right)^{-2} \iint_{Q} \boldsymbol{\mu}(\ell_{\theta,r}(10(\delta/\Delta)) \, d\theta dr \notag\\
& \stackrel{\eqref{form39}}{\geq} \sum_{Q \in \mathcal{Q}} \boldsymbol{\mu}(T(Q)) \stackrel{\eqref{form38}}{\gtrapprox_{\Delta}} |\mathcal{T}| \cdot \epsilon^{-1}(\delta/\Delta)\Delta^{\tau - t}. \end{align} 
On the other hand, one may use Theorem \ref{thm2a} to bound $\mathcal{I}_{10(\delta/\Delta)}$ from above:
\begin{align*} \mathcal{I}_{10(\delta/\Delta)}(\boldsymbol{\mu},\boldsymbol{\nu}) \lesssim \frac{\delta}{\Delta} \sqrt{I_{3 - \underline{t}}(\boldsymbol{\nu})I_{\underline{t}}(\boldsymbol{\mu})} \stackrel{\eqref{form28} + \eqref{form22}}{\lesssim} \frac{\delta}{\Delta} \sqrt{(\delta/\Delta)^{-(\sigma + 1)}|\mathcal{T}| \cdot \Delta^{-t}\mu(10Q)}. \end{align*} 
Combining this with the lower bound from \eqref{form30}, and rearranging, yields
\begin{displaymath} |\mathcal{T}| \lessapprox_{\Delta} \epsilon^{2} \Delta^{t - 2\tau} \cdot \left(\tfrac{\delta}{\Delta}\right)^{-(\sigma + 1)}\mu(10Q). \end{displaymath} 
Comparing this with the lower bound for $|\mathcal{T}|$ recorded in \eqref{form20}, and rearranging some more, one ends up with $\epsilon \gtrapprox_{\Delta} \Delta^{\tau - t}$. This means the same as
\begin{displaymath} \epsilon \geq \mathbf{C}^{-1}\log^{-\mathbf{C}}(1/\Delta)\Delta^{\tau - t} \end{displaymath}
for some constant $\mathbf{C} = \mathbf{C}(C,s,t,\tau) \geq 1$.
This however contradicts the choice of "$\epsilon$" at \eqref{form31}, choosing finally $\mathbf{A} > \mathbf{C}$. The proof of Theorem \ref{thm4} is complete. \end{proof}

\bibliographystyle{plain}
\bibliography{references}

\end{document}